%% file: ex_article.tex
\documentclass[onefignum,onetabnum]{siamart171218}

\input{ex_shared}

\ifpdf
\hypersetup{
  pdftitle={Solving Convolution-type Integral Equations using Preconditioned Neural Operators},
  pdfauthor={Raymond H. Chan and Lingfeng Li}
}
\fi

\myexternaldocument{ex_supplement}

\newsiamremark{remarks}{Remark}
\newsiamthm{assumption}{Assumption}

\begin{document}
\newcommand{\bpsi}{\boldsymbol\psi}
\newcommand{\B}{\mathcal{B}}
\newcommand{\bomega}{\boldsymbol\omega}
\newcommand{\divg}{\text{div}}
\newcommand{\relu}{\text{ReLU}}
\newcommand{\requ}{\text{ReQU}}
\newcommand{\SP}{\text{SP}}
\newcommand{\tb}[1]{\textcolor{blue}{#1}}
\newcommand{\tr}[1]{\textcolor{red}{#1}}
\newcommand{\tp}[1]{\textcolor{purple}{#1}}
\newcommand{\normu}{\lVert u\rVert_{\B(\Omega)}}

\maketitle

\begin{abstract}
  Convolution-type integral equations arise from various fields, \textit{e.g.}, finite impulse response filters in signal processing and deblurring problems in image processing. When solving these equations, conventional numerical methods, like the multigrid method, can only efficiently solve the low-frequency components in the error, but not the high-frequency components. In this paper, we apply neural operators to address this issue. By adopting a preconditioning approach, we propose a novel training strategy that trains neural operators to solve the high-frequency components efficiently. Then, we combine the neural operators with some classical iterative solvers, like the weighted Jacobi method, to obtain an efficient hybrid iterative algorithm for the integral equations. We analyze the generalization error of our training strategy and the convergence of the hybrid iterative algorithm. We test our algorithms on large-scale and ill-conditioned linear systems discretized from one- and two-dimensional convolution-type integral equations. Our proposed algorithm significantly outperforms the multigrid method and the preconditioned conjugate gradient method in both iteration numbers and computational time.
\end{abstract}



\section{Introduction}

Solving linear systems is one of the most fundamental tasks in scientific computing.  In general, there are two types of algorithms for solving linear systems: direct methods and iterative methods. Direct methods can find the exact solutions to linear systems using some factorization techniques, \textit{e.g.}, the LU decomposition method. Although direct methods can find the exact solutions in a finite number of operations, their computational costs increase very fast with the problem size, which makes it impractical to solve large-scale linear systems. Iterative methods, in contrast, solve linear systems by iteratively refining their solutions. They are particularly useful for large-scale problems because of their low computational costs compared to direct methods. The solution of typical iterative methods, such as the Jacobi method, the Gauss-Seidel method, and the Krylov subspace methods, would approach the exact solution as the iteration number increases. The convergence rate usually depends on the condition number of the coefficient matrix. For very ill-conditioned matrices, preconditioning techniques can help accelerate the convergence.  One important class of linear systems is the Toeplitz systems, which arise from many applications, especially in the solution of convolution-type integral equations, in signal processing, and in image processing \cite{chui1982application,grenander1958toeplitz,king1989digital}. Toeplitz matrices are matrices with constant diagonals. By exploring this special structure, many efficient preconditioners have been proposed \cite{chan1988optimal,chan1994circulant,tyrtyshnikov1992optimal}.  

In the following, we use the terminology in \cite{xu2017algebraic} to call the eigenvectors corresponding to large (respectively, small) eigenvalues of a matrix the {\it algebraic} high-frequency (respectively, low-frequency) eigenvectors. If the eigenvector is itself highly-oscillating (respectively, slowly-oscillating), we call the vector a {\it geometric} high-frequency (respectively, low-frequency) eigenvector.  Iterative methods, such as the Jacobi and Gauss-Seidel methods, are called {\it smoothers} in that they tend to reduce the algebraic high-frequency components (eigenvectors corresponding to large eigenvalues) in the error much faster than the algebraic low-frequency components (eigenvectors corresponding to small eigenvalues). \cite{xu2017algebraic}. 

For PDE problems, the algebraic high-frequency and low-frequency components of the discretized systems coincide with the geometric high-frequency and low-frequency components, respectively, \textit{i.e.}, they are highly oscillatory and slowly oscillatory, respectively. In such cases, we can apply the multigrid method to solve the linear systems efficiently \cite{hemker1984some}. The idea of multigrid methods is to reduce the geometric high-frequency error using smoothers such as the Jacobi method, and remove the geometric low-frequency error using coarse-grid corrections. The multigrid method is often considered the optimal method for PDE problems, particularly in terms of computational efficiency \cite{briggs2000multigrid}. 

Some recent studies borrow the idea from multigrid methods and replace the coarse-grid correction step with neural operator solvers \cite{hu2024hybrid,zhang2022hybrid, Zhang2024BlendingNO}. For example, the DeepONet-based neural operators \cite{lu2021learning}, which are designed for operator learning in continuous spaces, can reduce the geometric low-frequency error \cite{zhang2022hybrid,Zhang2024BlendingNO}. This type of structure is known to be frequency-biased \cite{xu2020frequency}, \textit{i.e.}, it tends to learn the geometric low-frequency components in the solution rather than the high-frequency components during training. Therefore, it can be combined with smoothers to solve PDE problems. 

There also exist other types of neural operators that learn operators in discrete spaces, like convolutional neural networks (CNN) and Fourier neural operators (FNO) \cite{li2021fourier}. In \cite{huang2022learning,song2025physics}, the authors use CNNs as multigrid smoothers and apply them with coarse-grid corrections to solve PDE problems. Typically, neural operators require a long time of training before they can be used in applications, so these types of methods are useful when a linear system needs to be solved multiple times for different right-hand-side vectors, \textit{e.g.}, when simulating the evolution of fluid dynamics or heat transfer \cite{bergheau2013finite}. Besides, there are also some neural operators that integrate the multigrid framework into the structure directly for solving linear systems \cite{chen2022meta,he2024mgno}.

This paper concerns problems where the algebraic high-frequency (respectively, low-frequency) components do not correspond to geometric high-frequency (respectively, low-frequency) components. This occurs in some convolution-type integral equations with smooth kernels \cite{chan1997multigrid}, where the corresponding linear systems are typically ill-conditioned Toeplitz or block-Toeplitz. For these systems, the multigrid method may fail since both smoothers and coarse-grid corrections tend to reduce the same algebraic high-frequency (geometric low-frequency) components in the error. Finding a good smoother to solve the algebraic low-frequency (geometric high-frequency) components is a challenging task for this type of problem \cite{chan1997multigrid}.  

In this work, we aim to solve the algebraic low-frequency components using neural operators and then combine the operators with smoothers to form a new hybrid iterative algorithm for the integral equations. 
The main contributions and findings are summarized as follows:
\begin{enumerate}
    \item We developed a new algorithm that solves algebraic low-frequency components using neural operators. The main difference between our method and previous work \cite{hu2024hybrid,zhang2022hybrid, Zhang2024BlendingNO,huang2022learning,song2025physics} is that our method is designed to address algebraic frequency components, whereas most previous work focuses on specific geometric frequency components.

    \item Inspired by the framework in \cite{zhang2022hybrid, Zhang2024BlendingNO}, we apply a hybrid iterative algorithm combining neural operators and classical smoothers to solve convolution-type integral equations.

    \item We conduct some theoretical analysis of the proposed algorithm. We first estimate the generalization error when the neural operator is a two-layer ReLU network and derive an upper bound in terms of the number of training samples, problem size, and the optimization error. We then prove the convergence of the hybrid iterative algorithm when the neural operator satisfies some reasonable conditions. Lastly, we analyze how the choice of training loss function would affect the convergence rate of different algebraic frequency components in the training error.

    \item We conduct numerical experiments to solve large-scale and ill-conditioned integral equations using the proposed method. Our method outperforms multigrid methods and preconditioned conjugate gradient (PCG) methods significantly, both in terms of the iteration number and evaluation time. Besides, our experiments show that the convergence rate of our method is robust against the change in problem size and condition number.
\end{enumerate}


\section{Preliminaries}
\subsection{Convolution-type integral equations and Toeplitz matrices}
In image processing and data analysis, we often need to solve integral equations of the form:
\begin{equation}
    \alpha R(u)(z)+\int_\Omega \mathcal{K}(z-z')u(z')dz'=f(z),\quad z\in\Omega. \label{eq:integral equation}
\end{equation}
where $u$ is a continuous image or signal, $\alpha>0$ is the regularization weight parameter, $\mathcal{K}$ is a smooth convolution kernel, $\Omega$ is a fixed domain, and $R(u)$ is the regularization term. When $\mathcal K$ is a smooth function, \textit{e.g.}, a Gaussian smoothing kernel, the coefficient matrix of $\mathcal K$ would be very ill-conditioned, so a regularization term $R$ is required \cite{beck2009fast}.  Equation (\ref{eq:integral equation}) is called the integral equation of the second kind when $R(u)=u$, and is called the integral equation of the first kind when $\alpha=0$ \cite{polyanin2008handbook}.

Discretizing (\ref{eq:integral equation}) leads to a linear system  $A_n\mathbf x=\mathbf y$
where $A_n\in\mathbb{R}^{n\times n}$ is the coefficient matrix, which usually has some special structures like Toeplitz or Block-Toeplitz-with-Toeplitz-Block (BTTB). In this work, we consider the case when $A_n$ is symmetric positive definite (SPD). Typical choices for the regularization $R(u)$ include the Tikhonov regularization $u$, isotropic diffusion $\Delta u$, and the total variation $-\nabla\cdot(\nabla u/|\nabla u|)$. When the weight parameter $\alpha$ is small, $A_n$ is typically ill-conditioned.  

The matrix $A_n$ is called Toeplitz if each diagonal of its coefficient matrices is constant. We denote the value of elements at the main diagonal, $k$-th supdiagonal, and $k$-th subdiagonal as $a_0$, $a_k$, and $a_{-k}$ respectively.
Toeplitz systems are an important class of linear systems, and their solutions can be found efficiently by leveraging the special structure. Many fast solvers, including direct methods \cite{ammar1988superfast,bitmead1980asymptotically} and iterative methods \cite{chan1992circulant,chan1988optimal,tyrtyshnikov1992optimal}, have been proposed for solving Toeplitz systems. 

A $2\pi$-periodic Lebesgue integrable function $f:\mathbb{R}\rightarrow\mathbb{R}$ is called the generating function of $A_n$ if its Fourier coefficients are the entries of $A_n$: $a_k=\frac{1}{2 \pi} \int_{-\pi}^\pi f(\theta) e^{-i k \theta} d \theta,$ where $\quad k=0, \pm 1, \pm 2, \ldots.$
Let $\lambda_1(A_n)\geq \lambda_2(A_n)\geq\dots\geq\lambda_n(A_n)$ be the eigenvalues of $A_n$ sorted in descending order. It is proven that the eigenvalues of $A_n$ can be bounded by the range of its generating function \cite{grenander1958toeplitz}:
\begin{lemma}
Let $f$ be the generating function of $A_n$. Then, 
\begin{equation}
\min_{\theta\in[0,2\pi]}f(\theta)<\lambda_n(A_n) \leq \lambda_1(A_n)<\max_{\theta\in[0,2\pi]}f(\theta). \label{bounds}
\end{equation}
\end{lemma} 

\subsection{Algebraic high/low frequency} \label{sec:algebraic_frequency}
The terms algebraic high/low frequency are commonly used in the analysis of algebraic multigrid methods \cite{xu2017algebraic}.  When $A_n$ is symmetric positive definite (SPD), all the eigenvalues are positive, and there exists a set of orthonormal eigenvectors $\{\mathbf v_i\}_{i=1}^n$ where $\mathbf v_i$ corresponds to $\lambda_i(A_n)$. The $\mathbf v_i$ corresponding to the large (resp. small) eigenvalues of $A_n$ are the algebraic high-frequency (resp. low-frequency) modes associated with $A_n$. Meanwhile, vectors can also be characterized as geometric high-frequency or geometric low-frequency based on the variation of their entries. Geometric high-frequency vectors refer to those vectors with rapidly oscillating entries, while geometric low-frequency vectors refer to those with smooth and slow oscillatory entries. Mathematically, for any given $\zeta_1,\zeta_2$ such that $\zeta_2\geq\zeta_1>0$, a vector $\mathbf v$ is called $\zeta_1-$algebraic low-frequency if \(\Vert\mathbf v\Vert_{A_nD_n^{-1}A_n}^2/\Vert\mathbf v\Vert_{A_n}^2\leq\zeta_1\),  and  $\zeta_2-$algebraic high-frequency if \(\Vert\mathbf v\Vert_{A_nD_n^{-1}A_n}^2/\Vert\mathbf v\Vert_{A_n}^2\geq\zeta_2\), where $D_n$ is the diagonal component of $A_n$ \cite{ruge1987algebraic,xu2017algebraic}. 

For any Toeplitz matrix $A_n$, we have \(\Vert \mathbf v\Vert_{A_nD_n^{-1}A_n}^2=\Vert \mathbf v\Vert_{A_n^2}^2/a_0\). We can show that
\[  \frac{1}{a_0}\Vert\mathbf v\Vert_{A_n^2}^2=\frac{1}{a_0}\Vert A_n^{1/2}(A_n^{1/2}\mathbf v)\Vert^2_2\in\left[ \frac{\lambda_n(A_n)}{a_0} \Vert \mathbf v\Vert^2_{A_n},\frac{\lambda_1(A_n)}{a_0} \Vert \mathbf v\Vert^2_{A_n}\right].\]
Thus, we need to choose $\zeta_1$ and $\zeta_2$ such that $\frac{\lambda_n(A_n)}{a_0}\leq\zeta_1\leq\zeta_2\leq\frac{\lambda_1(A_n)}{a_0}$. 
For simplicity,  in the following, we omit the explicit dependence of $\zeta_1$ (resp. $\zeta_2$) in $\zeta_1$-algebraic low frequency (resp. $\zeta_2$-algebraic high frequency) and simply refer to them as algebraic low frequency (resp. algebraic high frequency). 

{For some convolution-type integral equations, the algebraic high-frequency (resp. low-frequency) components of the coefficient matrices coincide with the geometric low-frequency (resp. high-frequency) components. In contrast, coefficient matrices of most PDE problems have the opposite spectral property. In Figure \ref{fig:eigenvectors} (top row),  we visualize the eigenvectors corresponding to the large and small eigenvalues of the coefficient matrix of an integral equation (\ref{eq:integral equation}) with $\mathcal K$ being a Gaussian smoothing kernel. We observe that the algebraic high-frequency (resp. low-frequency) components correspond to the geometric low-frequency (resp. high-frequency) components. For comparison, we also visualize the eigenvectors of the coefficient matrix of the Poisson equation in Figure \ref{fig:eigenvectors} (bottom row), where the algebraic high-frequency (resp. low-frequency) components correspond to the geometric high-frequency (resp. low-frequency) components.}

\begin{figure}
    \centering
    \includegraphics[width=.9\textwidth]{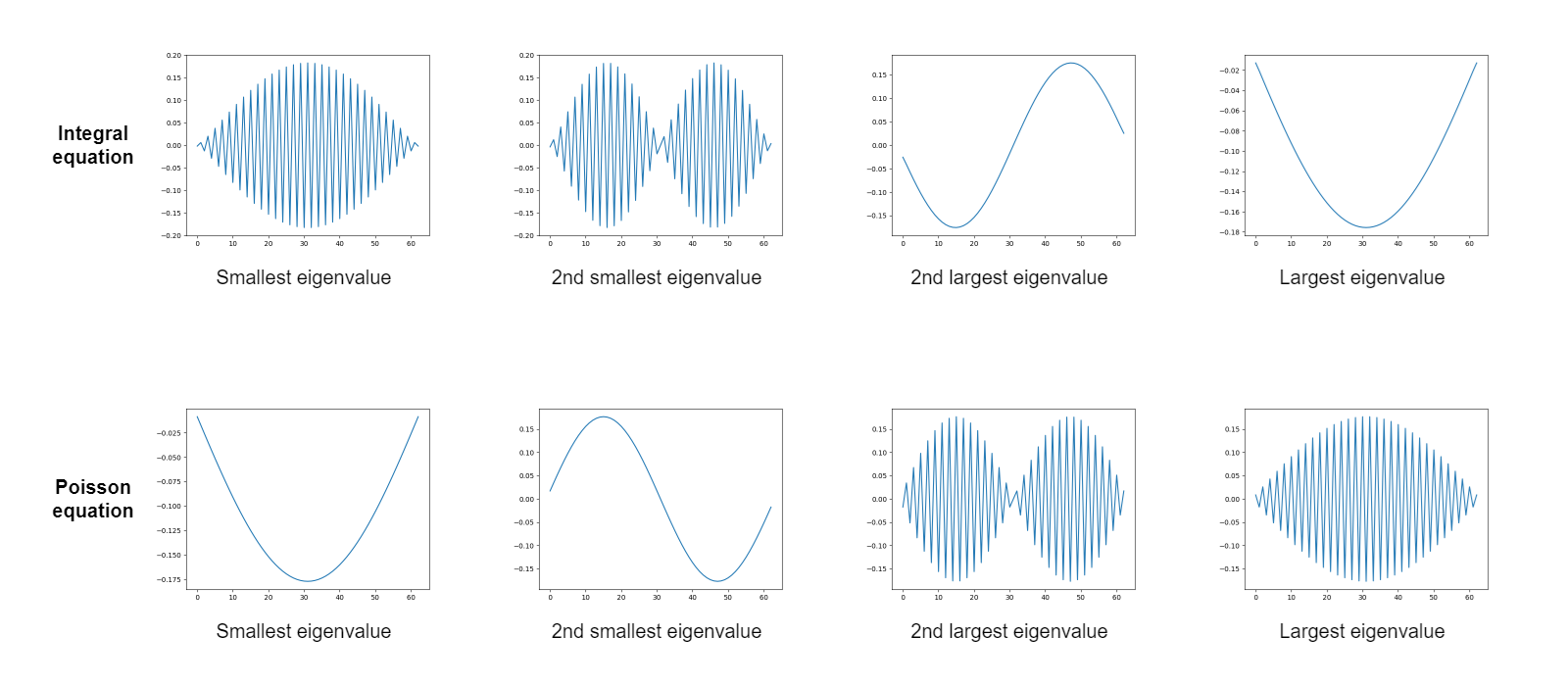}
    \caption{Eigenvectors corresponding to different eigenvalues of the linear system discretized from the integreal equation (top row) and the Poisson equation (bottom row).}
    \label{fig:eigenvectors}
\end{figure}

\subsection{Two-grid method} \label{sec:two-grid}
The algebraic multigrid method is a very efficient solver for solving linear systems. It consists of the smoothing steps and coarse-grid correcting steps. Here, we consider a simple two-grid method. 
The smoothing step is usually performed by some smoothers. We first define an iterative scheme for solving $A_n\mathbf x=\mathbf y$ 
\begin{equation}
    \mathbf x^{(\kappa+1)}=(I_n-M^{-1}_nA_n)\mathbf x^{(\kappa)}+M_n^{-1}\mathbf y,\quad \kappa=1,2,\dots, \label{eq:iterative_solver}
\end{equation}
where $\mathbf x^{(0)}$ is the initial vector, $I_n$ is the $n$-by-$n$ identity matrix, and $M_n$ is an $n$-by-$n$ matrix whose inverse can be evaluated efficiently. We further define $\mathcal{S}(A_n,\mathbf y,\mathbf x^{(0)},\kappa)$ as the procedure of solving $A_n\mathbf x=\mathbf y$ using $\kappa$ iterations of (\ref{eq:iterative_solver}) started from $\mathbf x^{(0)}$. Denote the iteration matrix as $S_n:=I_n-M_n^{-1}A_n$, then the error of $\mathbf x^{(\kappa)}$ can be written as
\begin{equation}
    \mathbf x^{(\kappa)}-\mathbf x=S_n(\mathbf x^{(\kappa-1)}-\mathbf x)=\dots=S_n^{\kappa}(\mathbf x^{(0)}-\mathbf x). \label{eq:iteration_matrix}
\end{equation}
Many classical iterative methods, like the damped Jacobi iteration, satisfy the smoothing condition \cite{xu2017algebraic,ruge1987algebraic}: for some $\sigma_1>0$ and any $\mathbf e\in\mathbb{R}^n$ 
\begin{equation}
    \Vert S_n\mathbf e\Vert_{A_n}^2\leq \Vert\mathbf e\Vert_{A_n}^2 - \sigma_1\Vert \mathbf e\Vert_{A_n^2}^2/a_0. \label{eq:smoothing_condition}
\end{equation}
From (\ref{eq:iteration_matrix}) and (\ref{eq:smoothing_condition}), if the error \(\mathbf x^{(\kappa-1)}-\mathbf x\) contains algebraic high-frequency modes, \textit{i.e.}, \(\Vert \mathbf x^{(\kappa-1)}-\mathbf x\Vert_{A_n}^2\) and \(\Vert \mathbf x^{(\kappa-1)}-\mathbf x\Vert_{A_n^2}^2/a_0\) are comparable in magnitude, after applying the smoothing operator \(S_n\), \(\Vert \mathbf x^{(\kappa)}-\mathbf x\Vert_{A_n}^2\) would be efficiently reduced. 

The correcting step for $A_n\mathbf x=\mathbf y$ is typically performed by a coarse grid correction:
$$ \hat{\mathbf x}=P_H^h(A_H)^{-1}R_h^H\mathbf y, $$
and the iteration matrix is $T_n:=I_n-P_H^h(A_H)^{-1}R_h^HA_n.$
Here, $P_H^h$ denotes the prolongation matrix that interpolates a vector on a coarse grid to a fine grid, $R_h^H$ denotes the restriction matrix that restricts a vector on a fine grid to a coarse grid, and $A_H:=R_h^HA_nP_H^h$ is the coarse grid coefficient matrix. 

The two-grid method solves the linear system $A_n\mathbf x=\mathbf y$ by alternatively applying the smoothing step and the correcting step (see Algorithm \ref{alg:two-grid}). By restricting a fine grid vector to a coarse grid, the geometric low-frequency modes are preserved. Therefore, solving the coarse grid problem can eliminate the geometric low-frequency modes efficiently. For most PDE problems, the geometric low-frequency modes are also algebraic low-frequency modes, so the coarse-grid correction step and the smoothing step are complementary to each other. The overall convergence of the two-grid iteration is guaranteed if for some $\sigma_2>0$ and any $\mathbf e\in\mathbb{R}^n$, \(T_n\) satisfies \cite{sun1997note}: 
\begin{equation}
    \Vert T_n\mathbf e\Vert_{A_n}^2\leq \frac{\sigma_2}{a_0}\Vert T_n\mathbf e\Vert_{A_n^2}^2 \label{eq:correcting_condition} \text{ and } \Vert T_n\mathbf e\Vert_{A_n}\leq\Vert\mathbf  e\Vert_{A_n}.
\end{equation} 
When $\sigma_2$ is sufficiently small, the first inequality implies that the error vectors after the correction step, \textit{i.e.}, $ T_n\mathbf e$, will not be algebraic low-frequency vectors (by the definition of algebraic low-frequency in Section \ref{sec:algebraic_frequency}). Therefore, it can be further solved efficiently by the smoothing step.
Then, by combining (\ref{eq:smoothing_condition}) and (\ref{eq:correcting_condition}), the two-grid method would converge in a factor of $1-\sigma_1/\sigma_2$:
\[\Vert S_nT_n\mathbf e\Vert_{A_n}^2\leq\Vert T_n\mathbf e\Vert_{A_n}^2-\sigma_1\Vert T_n\mathbf e\Vert_{A_n^2}^2/a_0\leq\Vert T_n\mathbf e\Vert_{A_n}^2-\frac{\sigma_1}{\sigma_2}\Vert T_n\mathbf e\Vert_{A_n}^2\leq(1-\frac{\sigma_1}{\sigma_2})\Vert\mathbf  e\Vert_{A_n}^2. \]

\begin{algorithm}
    \caption{Two-grid method}\label{alg:two-grid}
    \begin{algorithmic}
        \Require right-hand side vector $\mathbf y$ and the tolerance $\epsilon$
        \State $\tau=0$, and $\mathbf x_0=0$ 
        \While{$\Vert \mathbf y-A_n\mathbf x_\tau\Vert_2/\Vert \mathbf y\Vert_2\geq \epsilon$}
            \State $\hat{\mathbf x}_\tau=T_n\mathbf x_{\tau}+ P_H^h(A_H)^{-1}R_h^H\mathbf y$ \Comment{Correcting step}
            \State $\mathbf x_{\tau+1}=\mathcal{S}(A_n,\mathbf y,\hat{\mathbf x}_\tau,\kappa)$ \Comment{Smoothing step}
            \State $\tau=\tau+1$
        \EndWhile
    \State Return $\mathbf x_\tau$
    \end{algorithmic}
\end{algorithm}

When the algebraic low-frequency modes are not geometrically low-frequency, which is common for integral equations, the two-grid method may fail because the smoothing steps and the correcting steps both reduce the same high-frequency parts of the errors. Detailed experiments and analysis of the difficulties in solving integral equations using multigrid methods can be found in \cite{chan1997multigrid}. 

\subsection{Neural operators}
Neural operators are neural networks that can approximate mappings between function spaces. Once trained, a neural operator can approximate the solution of the learned mappings directly for various input conditions, which makes them particularly well-suited for solving linear or nonlinear systems \cite{lu2021learning,li2021fourier,wang2021learning,huang2022learning}.

For solving a linear system $A_n\mathbf x=\mathbf y$, where $\mathbf x$ and $\mathbf y$ are in $\mathbb{R}^n$, we need to construct a neural operator $\mathcal{N}_\theta:\mathbb{R}^n\rightarrow\mathbb{R}^n$ that can approximate $\mathbf x=A_n^{-1}\mathbf y$ by $\mathcal{N}_\theta(\mathbf y)$, where $\theta$ denotes the set of all learnable parameters in the network. Here, $\mathcal{N}_\theta$ can be various types of networks, like the fully connected networks, ResNet \cite{he2016deep}, U-Net \cite{ronneberger2015u}, and the FNO \cite{li2021fourier}. 
Typically, the training of neural operators can be data-driven or physics-informed. 

The data-driven method \cite{li2021fourier} directly minimizes the difference between the network prediction $\mathcal{N}_\theta(\mathbf y)$ and the true solution $\mathbf x=A_n^{-1}\mathbf y$: $\min_\theta \frac{1}{N}\sum_{i=1}^N\Vert\mathcal{N}_\theta(\mathbf y_i)-\mathbf x_i\Vert_2^2.$
Here $\{\mathbf x_i,\mathbf y_i=A_n\mathbf x_i\}_{i=1}^N$ is a randomly generated set of training samples. 

The physics-informed method \cite{goswami2022physics,li2024physics,wang2021learning} minimizes the residual of the problem: $\min_\theta \frac{1}{N}\sum_{i=1}^N\Vert A_n\mathcal{N}_\theta(\mathbf y_i)-\mathbf y_i\Vert_2^2.$
It is particularly useful when the true solution $\mathbf x_i=A_n^{-1}\mathbf y_i$ is not available. 

\section{A hybrid iterative method for integral equations}
As we mentioned in the last paragraph of Section \ref{sec:two-grid}, integral equations of convolution type are difficult to solve using multigrid methods. The main challenge is to find a suitable smoother to efficiently reduce algebraic low-frequency components, \textit{i.e.}, geometric high-frequency components. Motivated by the recent attempts to integrate neural operators into multigrid frameworks \cite{huang2022learning,zhang2022hybrid,Zhang2024BlendingNO}, we propose a novel framework combining classical smoothers and neural operators to solve convolution-type integral equations. The main idea is to design a new loss function for training neural operators so that they can solve the desired frequency components as expected. 

\subsection{Training of neural operators} \label{sect3.1}


Here, we consider a general neural network $\mathcal{N}_\theta:\mathbb{R}^n\rightarrow\mathbb{R}^n$ for approximating the solution operator $A_n^{-1}$ of the linear systems $A_n\mathbf x=\mathbf y$.  A general expected loss function for training $\mathcal{N}$ is:
\begin{equation}
\mathcal{L}_{\rm{exp}}(\theta):=\mathbb{E}_{\mathbf y\sim\mu(\mathbb{B}(\mathbb{R}^n))}(\Vert P_n(\mathcal{N}_\theta(\mathbf y)-\mathbf x)\Vert^2_2). \label{eq:expected_loss}
\end{equation}
Here, $\mathbb{B}(\mathbb{R}^n)$ is the unit sphere $\{\mathbf y\in\mathbb{R}^n| \Vert \mathbf y\Vert_2= 1\}$, $\mu$ is a probability distribution defined on $\mathbb{B}(\mathbb{R}^n)$, and $P_n$ is an $n\times n$ real matrix. Such a preconditioned loss function has been applied for physics-informed neural networks \cite{liu2024preconditioning} recently. To approximate $A_n^{-1}\mathbf y$ for $\Vert \mathbf y\Vert_2\neq 1$ using $\mathcal{N}_\theta$, we can first normalize $\mathbf y$ and then scale $\mathcal N_\theta(\mathbf y)$ by the same factor, \textit{i.e.}, $\mathcal{N}_\theta(\mathbf y/\Vert \mathbf y\Vert_2)\Vert \mathbf y\Vert_2$.  We suppose the matrix $P_n$ has a set of orthonormal eigenvectors $\{\mathbf u_i\}_{i=1}^n$ and the corresponding eigenvalues $\{\lambda_i(P_n)\}_{i=1}^n$. Then, the loss $\mathcal{L}_{\rm{exp}}$ can be written as
\begin{equation}
    \mathbb{E}_{\mathbf y}\left\Vert P_n\left(\sum_{i=1}^n(\mathbf e^\top \mathbf u_i)\mathbf u_i\right)\right\Vert^2_2=\mathbb{E}_{\mathbf y}\left\Vert \sum_{i=1}^n\lambda_i(P_n)(\mathbf e^\top \mathbf u_i)\mathbf u_i\right\Vert^2_2=\mathbb{E}_{\mathbf y}\sum_{i=1}^n\lambda_i(P_n)^2(\mathbf e^\top\mathbf  u_i)^2 \label{eqt6}
\end{equation}
where $\mathbf e:=\mathcal{N}(\mathbf y)-\mathbf x$  is the error vector and $\mathbf e^\top\mathbf  u_i$ represents the projection of  $\mathbf e$ on the unit eigenvector $\mathbf u_i$. From (\ref{eqt6}), we see that the loss function (\ref{eq:expected_loss}) is precisely equivalent to a weighted summation of $(\mathbf e^\top \mathbf u_i)^2$, $i=1,\dots,n$, and the weight for each term is the squared eigenvalues $\lambda_i^2(P_n)$. Therefore, the neural operator damps more on the algebraic high-frequency components associated with $P_n$ in the error $\mathbf e$ during the training process.

To train the neural operator to effectively solve algebraic low-frequency components associated with $A_n$, we therefore need to choose $P_n$ such that it has the opposite spectral properties compared with those of $A_n$, \textit{i.e.}, algebraic high-frequency components of $P_n$ are algebraically low-frequency components with respect to $A_n$. An ideal choice would be $P_n=A_n^{-1}$, since $A_n^{-1}$ has the same eigenvectors as $A_n$ and its eigenvalues are the reciprocals of $A_n$'s. However, calculating the inverse of $A_n$ directly is not practical as it is our original problem.  Here, we propose two practical choices for $P_n$:
\begin{enumerate}
    \item  Let $P_n=C(A_n)^{-1}$ where $C(A_n)$ denotes a circulant preconditioner of $A_n$ \cite{chan1988optimal}. For Toeplitz matrices, $C(A_n)^{-1}$ is a good approximation to $A_n^{-1}$, and it can be evaluated efficiently by the fast Fourier transform \cite{chan1996conjugate}. It has been proven that the eigenvalues of $C(A_n)^{-1}A_n$ are clustered around 1 \cite{chan1989spectrum}.
    \item Let $P_n=aI_n-A_n$ for some $a>0$. For the Toeplitz matrix $A_n$ with generating function $f(\theta)$, we can choose $a$ such that $a\geq \max_{[0,2\pi]}f(\theta)$, see (\ref{bounds}). In this case, an eigenvector of $A_n$ corresponding to the eigenvalue $\lambda_i$ is also an eigenvector of $P_n$ corresponding to the eigenvalue $a-\lambda_i>0$.
\end{enumerate}
Using either way to construct $P_n$, we will be able to train a neural operator that can solve the desired frequency components. We note that when $A_n$ is SPD, $P_n$ is also SPD.

\begin{remarks}
    Here, the $P_n$ does not have to be a good preconditioner in the sense of numerical linear algebra, \textit{i.e.}, we do not require $P_n^{-1}A_n\approx I_n$. Instead, we only require $P_n$ has a reverse spectrum with $A_n$, making our method applicable even to problems without good preconditioners.
\end{remarks}

We demonstrate a simple example here to further illustrate the effect of $P_n$ on neural operator training. We consider a discretized one-dimensional Poisson equation whose eigenvectors are shown in Figure \ref{fig:eigenvectors}.
We first construct a vector $\mathbf x_0=\{-\cos(4\pi t/64)+\cos(20\pi t/64)-\cos(60\pi t/64)\}_{t=0}^{64}$ and compute $\mathbf y_0=A_n\mathbf x_0$. This vector $\mathbf x_0$ consists of three different frequency components. Then, we initialize a neural operator $\mathcal{N}_\theta$ and train it by $ \min_\theta \Vert P_n(\mathcal{N}_\theta(\mathbf y_0)-\mathbf x_0)\Vert^2_2. $
We choose $P_n=A_n$ and $P_n=C(A_n)^{-1}$ respectively, and compare the predictions of $\mathcal{N}_\theta$. In Figures \ref{fig:soln_compare} and \ref{fig:err_freq_compare}, we plot the neural operator predictions and the magnitude spectra of the errors after different numbers of training epochs. When $P_n=A_n$,   $\mathcal{N}_\theta$ solves the algebraic high-frequency (the same as geometric high-frequency) part associated with $A_n$  effectively, and $\mathcal{N}_\theta$ performs like a smoother that reduces the geometric high-frequency errors. On the contrary, when $P_n=C(A_n)^{-1}$,  $\mathcal{N}_\theta$ solves the algebraic high-frequency (the same as geometric low-frequency) part associated with $C(A_n)^{-1}$  effectively, and  $\mathcal{N}_\theta$ performs like a corrector that reduces geometric low-frequency errors.  

\begin{figure}[ht]
    \centering
    \subfigure[$P_n=A_n$]{\includegraphics[width=0.45\textwidth]{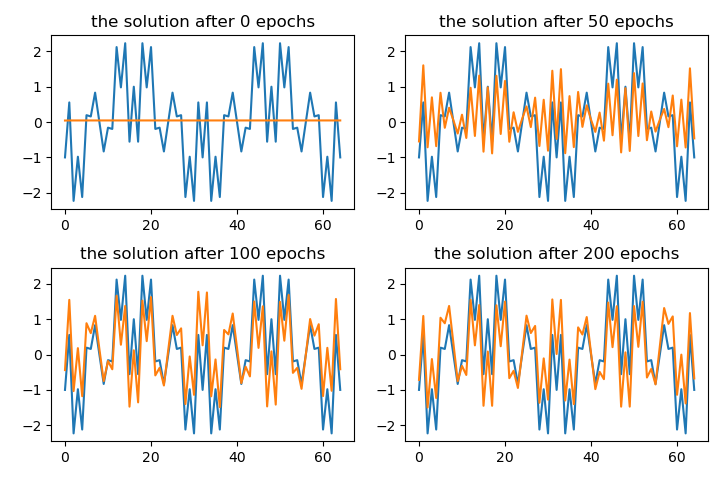}} 
    \subfigure[$P_n=C(A_n)^{-1}$]{\includegraphics[width=0.45\textwidth]{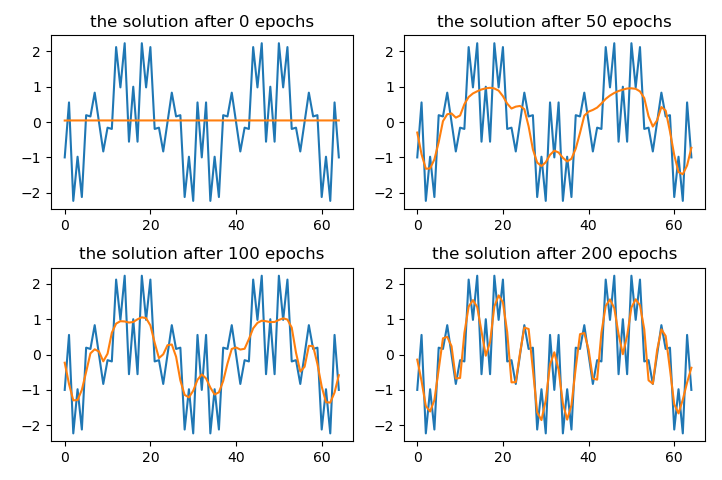}} 
    \caption{(a) and (b) shows the prediction $\mathcal{N}_\theta(\mathbf y_0)$ after different numbers of training epochs when $P_n=A_n$ and $P_n=C(A_n)^{-1}$ respectively. Here, the blue curves represent $\mathbf x_0$ and the orange curves represent $\mathcal{N}_\theta(\mathbf y_0)$.}
    \label{fig:soln_compare}
\end{figure}

\begin{figure}[ht]
    \centering
    \subfigure[$P_n=A_n$]{\includegraphics[width=0.45\textwidth]{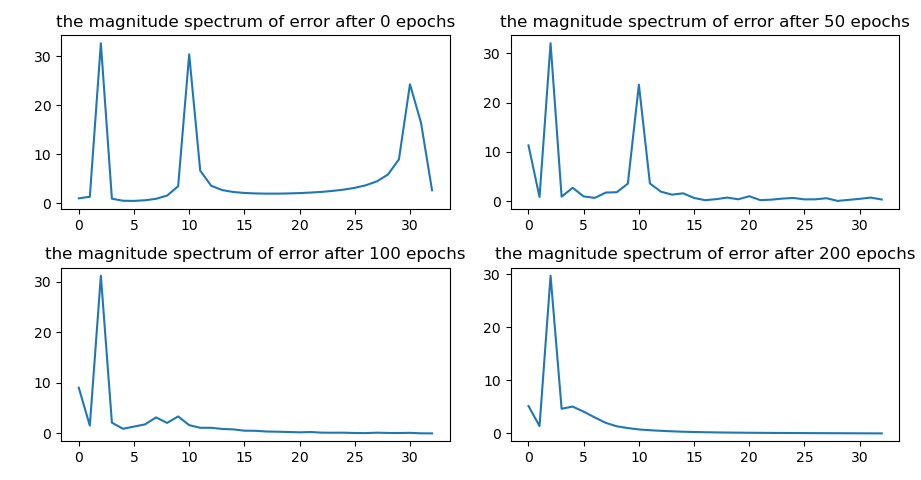}} 
    \subfigure[$P_n=C(A_n)^{-1}$]{\includegraphics[width=0.45\textwidth]{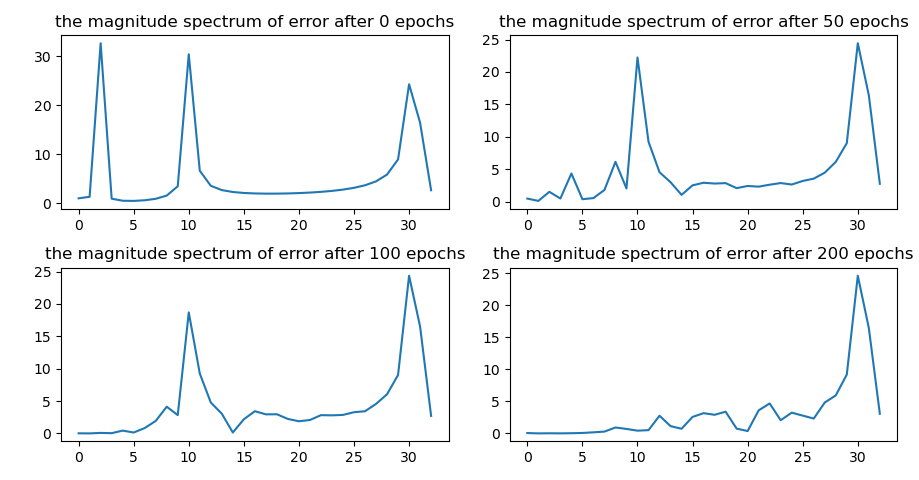}} 
    \caption{(a) and (b) show the Fourier magnitude spectra of the error after different numbers of training iterations when $P_n=A_n$ and $P_n=C(A_n)^{-1}$ respectively. The $x$-axis corresponds to the geometric frequency. }
    \label{fig:err_freq_compare}
\end{figure}

\subsection{The modified two-grid method}

With a well-trained neural operator solver, we can solve the linear system $A_n\mathbf x=\mathbf y$ by replacing the coarse-grid correction in the two-grid method (Algorithm \ref{alg:two-grid}) by a neural operator correction step. The resulting algorithm is shown in Algorithm \ref{alg:hybrid_iterative}. The overall structure of Algorithm \ref{alg:hybrid_iterative} is similar to the two-grid method. During each iteration, we first apply $\kappa_1$ steps of a classical smoother $\mathcal S$ for pre-smoothing. Then we use our trained neural networks developed in Section \ref{sect3.1} to correct the iterant $\mathbf x^{(\kappa+1)}$ in (\ref{eq:iterative_solver}). Finally, we use $\kappa_2$ steps of $\mathcal S$ for post-smoothing. In our algorithm, the smoothing steps are used to reduce algebraic high-frequency error, while the neural operator is used to reduce algebraic low-frequency error. 

We note that the HINTS method proposed in \cite{zhang2022hybrid,Zhang2024BlendingNO} also replaces the coarse-grid correction in the two-grid method by a neural operator solver DeepONet \cite{lu2021learning}, which tends to fit the geometric low-frequency in the target function first \cite{rahaman2019spectral}. This phenomenon is probably due to the multi-layer perceptron (MLP) module, which has been proven to favor geometric low-frequency \cite{cao2021towards}, in the structure of DeepONet. Thus, \cite{zhang2022hybrid,Zhang2024BlendingNO} utilize this spectral bias property and train DeepONets to learn geometric low-frequency. The method is good for PDE problems where the smoothers solve the algebraic high-frequency parts of the solution and the DeepONet solves the algebraic low-frequency parts of the solution (which is the same as the geometric low-frequency part), see Figure \ref{fig:eigenvectors}.

However, for integral equations, the geometric low-frequency part is also the algebraic high-frequency part, see Figure \ref{fig:eigenvectors}. Hence in the HINTS method, both the smoother and DeepONet will only solve the algebraic high-frequency part of the solution, and therefore the method will not be good for integral equations. Our neural operator in Section \ref{sect3.1} is designed to exactly recover the algebraic low-frequency part of the solution efficiently.

\begin{algorithm}
\caption{}\label{alg:hybrid_iterative}
\begin{algorithmic}
\Require right-hand side vector $\mathbf y$, number of pre-smoothing steps $v_1$, number of post smoothing steps $v_2$, and the tolerance $\epsilon>0$. 
\State Set $\tau=0$ and $\mathbf x_0=0$
\While{$\Vert \mathbf y-A_n \mathbf x_\tau\Vert_2/\Vert \mathbf y\Vert_2 \geq \epsilon$}
\State $\mathbf x_{\tau}^{(1)}=\mathcal{S}(A_n,\mathbf y,\mathbf x_\tau,\kappa_1)$  \Comment{Pre-smoothing steps}
\State $\mathbf r_\tau=(\mathbf y-A_n\mathbf  x_{\tau}^{(1)})/Z_\tau$ where $Z_\tau=\Vert \mathbf y-A_n \mathbf x_{\tau}^{(1)}\Vert_2$\Comment{Normalization step}
\State $\mathbf x_{\tau}^{(2)}=\mathcal{N}_\theta(\mathbf r_\tau)Z_\tau+\mathbf x_{\tau}^{(1)}$ \Comment{Neural operator correction step}
\State $\mathbf x_{\tau+1}=\mathcal{S}(A_n,\mathbf y,\mathbf x^{(2)}_\tau,\kappa_2)$ \Comment{Post-smoothing steps}
\State $\tau=\tau+1$
\EndWhile
\State Return $\mathbf x_\tau$
\end{algorithmic}
\end{algorithm}

\section{Analysis of proposed methods}
In this section, we give some theoretical analysis of our proposed methods. We first derive a generalization error estimate for the neural operator training (\ref{eq:expected_loss}). Then, we analyze the convergence of Algorithm \ref{alg:hybrid_iterative} under similar assumptions as the two-grid method \cite{ruge1987algebraic}. Finally, we study how the choice of $P_n$ in the loss function (\ref{eq:expected_loss}) would affect the convergence of different frequency components in the training error.

\subsection{Generalization errors analysis of neural operators}
We aim to analyze the generalization error when training the neural operator $\mathcal{N}_\theta$. The generalization error reflects how well the model can generalize to new, unseen data, rather than just memorizing the training data. In this section, we derive an analytical upper bound for the generalization error of $\mathcal{N}_\theta$ when $\mathcal{N}_\theta$ is defined as a two-layer ReLU network:
\begin{equation}    \mathcal{N}_\theta(\mathbf x):=\sum_{j=1}^2c_j\text{ReLU}(W_j\mathbf x),\label{eq:two_layer_network}
\end{equation}
where $c_j\in\mathbb{R}$, $W_j\in\mathbb{R}^{n\times n}$, $j=1,2$, and $\text{ReLU}(z)=\max\{z,0\}$.  {Two-layer networks are commonly used in the theoretical study of neural network methods \cite{lu2021priori,li2024priori,weinan2019priori} because of their simplicity.} We assume the parameter set $\theta$ is controlled by an $n$-dependent positive number $C_n$:
$$ \theta\in\Theta(C_n):=\left\{(c_1,c_2,W_1,W_2)\bigg| |c_j|\leq C_n, \Vert W_j\Vert_2\leq 1,j=1,2\right\}.  $$
Since $\text{ReLU}(cz)=c\text{ReLU}(z)$ for any $c>0$, there is no loss of generality of assuming $\Vert W_j\Vert_2\leq 1$. 

In practice, we are not able to precisely evaluate the expected loss $\mathcal{L}_{\rm{exp}}(\theta)$, so we approximate it with a discrete quadrature instead. Let $\{\mathbf  y_i\}_{i=1}^N$ be a set of right-hand-side vectors following a probability distribution $\mu$, which is defined on the unit sphere $\mathbb{B}(\mathbb{R}^n)$, and $\{\mathbf  x_i=A_n^{-1}\mathbf  y_i\}_{i=1}^N$ be the set of ground true solutions. Here, $N$ is the number of training samples. In practice, we first sample $\mathbf x_i$ from some random distribution and then compute $\mathbf y_i=A_n\mathbf x_i$.  Then, we train the network by minimizing the empirical loss function:
\begin{equation}
    \mathcal{L}^N_{\rm{emp}}(\theta):=\frac{1}{N}\sum_{i=1}^N\Vert P_n(\mathcal{N}_\theta(\mathbf y_i)-\mathbf x_i)\Vert_2^2. \label{eq:empirical_loss}
\end{equation} 
Here $\mathcal{L}^N_{\rm{emp}}(\theta)$ is a discrete quadrature approximation to $\mathcal{L}_{\rm{exp}}(\theta)$. 
We also define an error measurement for $\mathcal{N}_\theta$ as 
$$ \Vert \mathcal{N}_\theta-A_n^{-1}\Vert^2_{A_n,\mu}:=\mathbb{E}_{\mathbf y\sim\mu}\Vert\mathcal{N}_\theta(\mathbf y)-A_n^{-1}\mathbf y\Vert^2_{A_n}=\int_{\mathbf y\sim\mathbb{R}^d}\Vert\mathcal{N}_\theta(\mathbf y)-A^{-1}\mathbf y\Vert^2_{A_n}\mu(\mathbf y)d\mathbf y.$$
Then, the estimation of the generalization error $\Vert \mathcal{N}_\theta-A_n^{-1}\Vert^2_{A_n,\mu}$ is given below.
\begin{definition} \label{def:delta_minimizer}
    A vector $\mathbf x^*$ is a {\bf $\delta$-minimizer} of a function $f(\mathbf x)$ if $f(\mathbf x^*)\leq \delta+\inf_{\mathbf x}f(\mathbf x)$.
\end{definition}
In Definition \ref{def:delta_minimizer}, $\delta$ bounds the deviation between $f(\mathbf x^*)$ and the global minimum, so $\delta$ is a measure of the optimization error.
\begin{theorem} \label{theorem:generalization_error}
    Let $\{\mathbf y_i\}_{i=1}^N$ be a set of training data randomly sampled from a distribution defined on the unit ball $\mu(\mathbb{B}(\mathbb{R}^n))$, $\mathcal{N}_\theta$ is a two-layer network defined as (\ref{eq:two_layer_network}), and $\theta^*$ is a $\delta$-minimizer of the empirical loss:
    $$ \underset{\theta\in\Theta(C_n)}{\min} \mathcal{L}^N_{\rm{emp}}(\theta) $$
    where $C_n\geq\Vert A_n^{-1}\Vert_2$. 
    Then, the following inequality holds with probability at least $1-\epsilon$:
    $$ \Vert \mathcal{N}_{\theta^*}-A_n^{-1}\Vert^2_{A_n,\mu}\leq\Vert A_nP_n^{-2}\Vert_2\left( \frac{108nC_n^2\Vert P_n\Vert_{F}^2}{\sqrt{N}}+9C_n^2\Vert P_n\Vert_2^2\sqrt{\frac{2\log(2/\epsilon)}{N}}+\delta\right),$$
    where $\Vert\cdot\Vert_{F}$ denotes the matrix Frobenius norm.
\end{theorem}
The proof of Theorem \ref{theorem:generalization_error} is given in Appendix \ref{sec:generalization_error_proof}.
This theorem implies that, for any fixed $n$, with high probability the generalization error will converge to 0 as $N$ goes to infinity and $\delta$ goes to 0. Here, $\delta$ corresponds to the optimization error when minimizing (\ref{eq:empirical_loss}).

\subsection{Convergence analysis of Algorithm \ref{alg:hybrid_iterative}}
The convergence of conventional two-grid methods has been well studied in the literature \cite{ruge1987algebraic,sun1997note} based on reasonable assumptions on the smoothing step and the correcting step. We aim to conduct a similar analysis for our proposed hybrid iterative methods under similar assumptions.
We focus on a simple case of Algorithm \ref{alg:hybrid_iterative} when $\kappa_1=0$ and $\kappa_2=1$, \textit{i.e.}, only one post-smoothing step and no pre-smoothing step. We can prove the following estimate by making reasonable assumptions on the neural operators $\mathcal{N}_\theta$ similar to assumptions made for coarse-grid correction in \cite{ruge1987algebraic}.

\begin{theorem}\label{theorem:error_estimate}
    Suppose $\mathbf y\neq 0$ be a vector in $\mathbb{R}^n$, $\mathcal{N}_\theta$ be a neural network solver, and the iterative solver $\mathcal{S}$ (defined as (\ref{eq:iterative_solver}) with iteration matrix $S_n$) satisfies the smoothing condition (\ref{eq:smoothing_condition}) for some $\delta_1>0$. Let $\mathbf x_\tau$, $\tau=0,1,\dots$, be a sequence generated by Algorithm \ref{alg:hybrid_iterative} with $\kappa_1=0$ and $\kappa_2=1$. 
    Then, the following inequality holds
    $$ \Vert \mathbf x_{\tau+1}-A_n^{-1}\mathbf y\Vert_{A_n}^2\leq \left(1-\frac{\delta_1}{\delta_2}\right)\delta_3\Vert \mathbf x_{\tau}-A_n^{-1}\mathbf y\Vert_{A_n}^2 $$
    if $\mathcal{N}_\theta$ satisfies:
    \begin{enumerate}
        \item correcting condition: $\Vert \mathcal{N}_\theta(\mathbf r)-A_n^{-1}\mathbf r\Vert^2_{A_n}\leq \frac{\delta_2}{a_0}\Vert \mathcal{N}_\theta(\mathbf r)-A_n^{-1}\mathbf r\Vert^2_{A^2_n}$ for some $\delta_2>0$.
        \item stability condition: $\Vert \mathcal{N}_\theta(\mathbf r)-A^{-1}_n\mathbf r\Vert^2_{A_n}\leq \delta_3\Vert A^{-1}_n\mathbf r\Vert^2_{A_n}$ for some $\delta_3>0$.
    \end{enumerate}
    Here, $\mathbf r=\frac{\mathbf y-A_n\mathbf x_\tau}{\Vert \mathbf y-A_n\mathbf x_\tau\Vert_2}$.
\end{theorem}
The proof is given in Appendix \ref{sec:proof_error_estimate}.

Consequently, the iteration of Algorithm 2 (with $\kappa_1=0$ and $\kappa_2=1$) would converge if $(1-\delta_1/\delta_2)\delta_3<1$. Generally, it would be difficult to guarantee that the correcting condition and stability condition stated in Theorem \ref{theorem:error_estimate} hold all the time.
Using the Markov's inequality \cite{lin2010probability} and Theorem \ref{theorem:generalization_error}, we can show the stability condition would hold with high probability for well-trained $\mathcal{N}_\theta$.

\begin{corollary} \label{cor_stability}
    Given any confidence level $\epsilon>0$ and $\mathbf r\sim\mu(\mathbb{B}(\mathbb{R}^n))$. Let $\mathcal{N}_{\theta}$ be a two-layer network, and $\theta^*$ is a $\delta$-minimizer of the empirical loss:
    $$ \underset{\theta\in\Theta(C_n)}{\min} \mathcal{L}^N_{\rm{emp}}(\theta) $$
    where $C_n\geq\Vert A_n^{-1}\Vert_2$. 
    Then, the stability condition
    $$ \Vert \mathcal{N}_{\theta^*}(\mathbf r)-A_n^{-1}\mathbf r \Vert^2_{A_n}\leq \delta_3\Vert A_n^{-1}\mathbf r \Vert^2_{A_n} $$
    holds with probability at least $1-\epsilon$ if $N$ is sufficiently large and $\delta$ is sufficiently small.
\end{corollary}
The proof is given in Appendix \ref{sec:proof_stability}.
For the correcting condition in Theorem \ref{theorem:error_estimate}, we will verify it numerically for some convolution-type integral equations in Section \ref{sec:exp}.

\subsection{The effect of $P_n$ on the training dynamics}
In this part, we would like to study how the choice of $P_n$ in the training loss (\ref{eq:expected_loss}) would affect the training dynamics of our neural operators. More specifically, we study how the convergence rate of different algebraic frequency components in training error depends on the choice of $P_n$.

Let $\{\mathbf v_i\}_{i=1}^n$ be a set of orthonormal eigenvectors of $A_n$ with corresponding eigenvalues $\{\lambda_i\}_{i=1}^n$ and $\lambda_i$s are sorted in descending order: $\lambda_1\geq\lambda_2\geq\dots\lambda_n>0$. Then, for each $\mathbf y_j$, the projection of the prediction error $\mathcal{N}_\theta(\mathbf y_j)-\mathbf x_j$ on to each $\mathbf v_i$ is
\[ f_i^j(\theta):=(\mathcal{N}_\theta(\mathbf y_j)-\mathbf x_j)^\top \mathbf v_i.\]
Here, $f^j_i$ indicates the $i$-th algebraic frequency component of the prediction error on the $j$-th sample. We also defined the $i$-th algebraic frequency components in prediction errors over the entire training set by $F_i(\theta):=\frac{1}{N}\sum_{j=1}^N(f^j_i)^2$. In the previous section, we proposed that $P_n$ can be defined as $aI_n-A_n$ for some $a>\lambda_1$. Next, we will study how $F_i(\theta)$ would be affected by $a$.

Since the error $\mathcal N_\theta(\mathbf y_j)-\mathbf x_j$ can be decomposed by $\sum_{i=1}^n f_i^j(\theta)\mathbf v_i$. Then, the loss function $\mathcal L^N_{\text{emp}}(\theta)$ can be written as
\begin{align*}
\mathcal L^N_{\text{emp}}(\theta)&=\frac{1}{N}\sum_{j=1}^N\Vert P_n(\mathcal N_\theta(\mathbf y_j)-\mathbf x_j)\Vert_2^2\\
&=\frac{1}{N}\sum_{j=1}^N\left\Vert (aI_n-A_n)\left(\sum_{i=1}^n f_i^j(\theta)\mathbf v_i\right)\right\Vert_2^2\\
&=\frac{1}{N}\sum_{j=1}^N\left\Vert \sum_{i=1}^n f_i^j(\theta)(a-\lambda_i)\mathbf v_i\right\Vert_2^2\\
&=\frac{1}{N}\sum_{j=1}^N\sum_{i=1}^n f_i^j(\theta)^2(a-\lambda_i)^2\\
&=\sum_{i=1}^n(a-\lambda_i)^2\frac{1}{N}\sum_{j=1}^Nf_i^j(\theta)^2\\
&=\sum_{i=1}^n(a-\lambda_i)^2F_i(\theta).
\end{align*}
Consequently, for each $F_i(\theta)$, we have an upper bound $F_i(\theta)\leq\frac{\mathcal L^N_{\text{emp}}(\theta)}{(a-\lambda_i)^2}$ during training. We can compare the upper bound of $F_1$ and $F_n$
\[\frac{\mathcal L^N_{\text{emp}}(\theta)/(a-\lambda_1)^2}{\mathcal L^N_{\text{emp}}(\theta)/(a-\lambda_n)^2}=\frac{(a-\lambda_n)^2}{(a-\lambda_1)^2}.\]
When $a\rightarrow\lambda_1^+$, the ratio $\frac{(a-\lambda_n)^2}{(a-\lambda_1)^2}$ converges to $+\infty$, which implies that the algebraic low-frequency components in the training error have significantly lower bounds than the algebraic high-frequency components. Consequently, we expect that algebraic low-frequency components would decay much faster than algebraic high-frequency components during training.

We will further verify it by the following numerical examples. We train a two-layer fully-connected network by minimizing the empirical loss $\mathcal{L}_{\rm{emp}}^N$ using the gradient descent method, where the coefficient matrix $A_n$ is discretized from a one-dimensional integral equation (the construction of $A_n$ is the same as (\ref{eq:discretized_system_tik_1d}) in Section \ref{sec:exp_1d_tik}) The step size is of order $O(a^{-1})$. Here, we use $n=32$ and the maximum eigenvalue of $A_n$ is smaller than but close to $1.1$. Then, we visualize the change of $F_i(\theta^{\tau})$ during the first 50 time steps in Figure \ref{fig:training_dynamics} for $a=1.1$ and $a=1.5$. In Figure \ref{fig:training_dynamics}, $F_1$, $F_{n/2}$, and $F_n$ represent the algebraic high, medium, and low frequency components in the training error. When $a=1.1$, we observe that $F_n$ and $F_{n/2}$ decay fast while $F_1$ remains almost unchanged, which implies that $\mathcal N_\theta$ is focusing on solving the algebraic low frequency rather than the high frequency. When $a=1.5$, the convergence rates of different frequency components become comparable, indicating that $\mathcal N_\theta$ does not focus on any part of frequency specifically.

\begin{figure}[ht]
    \centering
    \subfigure[$a=1.1$]{\includegraphics[width=0.4\textwidth]{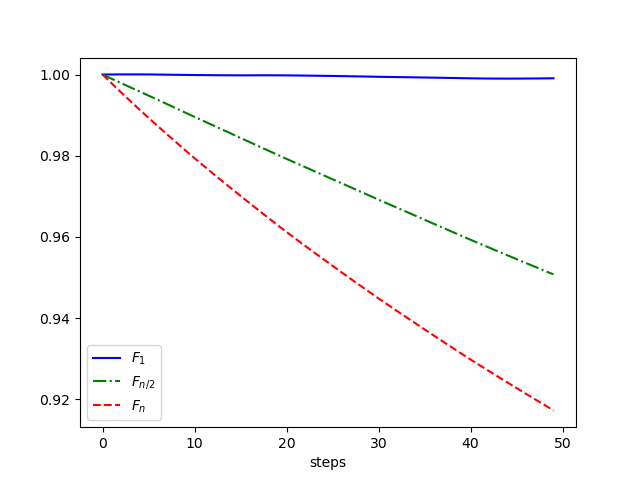}} 
    \subfigure[$a=1.5$]{\includegraphics[width=0.4\textwidth]{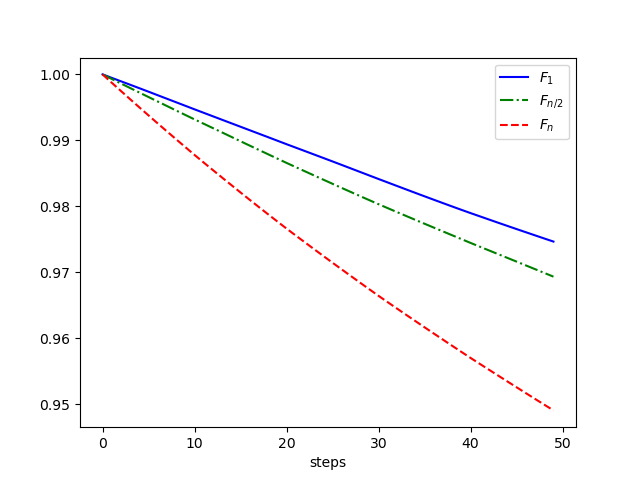}} 
    \caption{Figures of $F_i(\theta^\tau)/F_i(\theta^0)$, $\tau=0,1,\dots$, during training for different choices of $a$. }
    \label{fig:training_dynamics}
\end{figure}

\section{Numerical experiments} \label{sec:exp}
\subsection{Training setting}
The training of all neural operators in this section is conducted on the NVIDIA RTX A6000 GPU with 48 GB of memory. 
For one-dimensional problems, we train the neural operators for 1,000 epochs with a learning rate of $10^{-4}$ and a batch size of 100. For two-dimensional problems, we train for 2,000 epochs with a learning rate of $10^{-4}$ and a batch size of 10. Details of the neural operator structures are given in Appendix \ref{sec:network_structure}. In all experiments, we set $\kappa_2=0$ in Algorithm \ref{alg:hybrid_iterative}, {\it i.e.}, we only do the pre-smoothing step. In each batch, we randomly generate training samples $\mathbf x$ from a uniform distribution between $[0,1]$ (using the numpy.random.rand function) and compute the right-hand side vector $\mathbf y=A_n\mathbf x$. To improve the long-term stability of our methods, we apply a strategy similar to \cite{huang2022learning} during training. Specifically, after generating random training sample pairs $(\mathbf x,\mathbf y)$, we apply $\hat\kappa$ iterations of Algorithm \ref{alg:hybrid_iterative} with the latest updated $\mathcal N_\theta$. Then, we calculate the remaining error $\mathbf e$ and the residual $\mathbf r=A_n \mathbf e$. The pair $(\mathbf r,\mathbf e)$ is further used as training data. We visualize the training process in Figure \ref{fig:training}. The number $\hat\kappa$ is randomly chosen from 1 to 10.

\begin{figure}[h]
    \centering
    \includegraphics[width=0.7\textwidth]{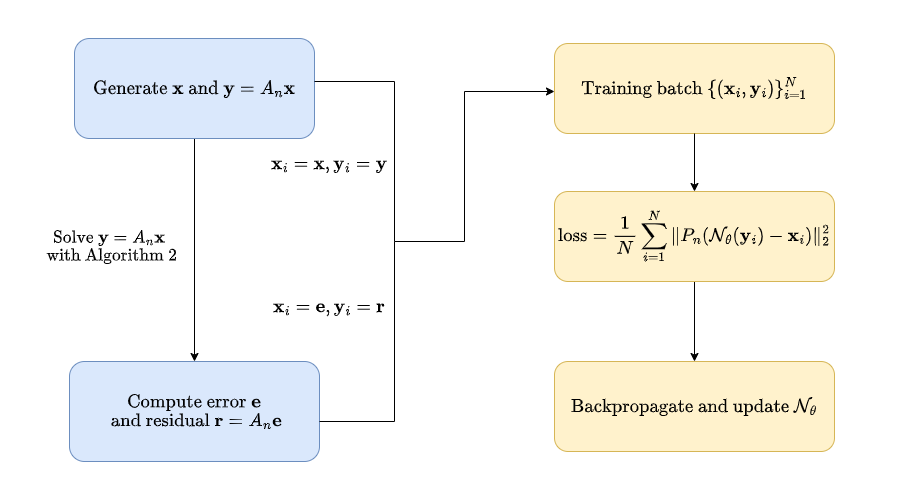}
    \caption{Training process of one batch}
    \label{fig:training}
\end{figure}

\subsection{One-dimensional integral equations} \label{sec:exp_1d}
\subsubsection{\textbf{Tikhonov regularized integral equation}} \label{sec:exp_1d_tik}
We consider (\ref{eq:integral equation}) where $R(u)=u$, $\mathcal{K}$ is the Gaussian smoothing kernel, and $\Omega=[0,1]$. The $R(u)=u$ corresponds to the gradient of the Tikhonov regularization term $\Vert u\Vert_{L^2(\Omega)}^2$.
The coefficient matrix of the discretized system is
\begin{equation}
    A_n=\alpha I_n+K_n \label{eq:discretized_system_tik_1d}
\end{equation}
where $\alpha>0$ is the weight of regularization, $I_n$ is the identity matrix, and $K_n$ is a discrete convolution matrix of a Gaussian smoothing kernel. We set the standard deviation of the Gaussian kernel to be 1.5. Then $A_n$ is a Toeplitz matrix. During training, we define $P_n=(\alpha+1) I_n-A_n$ in the loss function (\ref{eq:expected_loss}). Since the sum of each row of $K_n$ is less than or equal to 1, the maximum eigenvalue of $K_n$ is less than 1, which is guaranteed by the Gershgorin circle theorem \cite{golub2013matrix}. Then, the maximum eigenvalue of $A_n$ is less than $1+\alpha$. 
During the evaluation, we solve the system $A_n\mathbf x=\mathbf y$ using our proposed algorithm (Algorithm \ref{alg:hybrid_iterative} with $\kappa_1=5$, $\kappa_2=0$, and the pre-smoothing step is performed by the weighted Jacobi iterations with weight $0.3$) until the relative residual error is less than $10^{-10}$. For comparison, we also applied the two-grid method to solve this equation. Our two-grid method follows Algorithm \ref{alg:two-grid} with $\kappa=5$ and uses the weighted Jacobi iteration with weight $0.3$ to perform the smoothing step. We test all methods on different $\alpha$ and $n$. We also note that the number of required iterations is almost the same for both the two-grid method and the more advanced multigrid methods. This is because the dominant components in the error are the highest geometric frequency (or equivalently, the lowest algebraic frequency), which is only solvable on the finest grid. Consequently, convergence is primarily affected by the number of smoothing steps applied at the finest level, rather than the specific multigrid scheme. The condition numbers of the system $A_n$ corresponding to different $\alpha$ and $n$ are listed in Table \ref{tab:condNum}.
The \textbf{training time} for our neural operators is about \textbf{$10$min},  \textbf{$12$min}, and \textbf{$16$min} when $n=256$, $512$, and $1024$ respectively. For each pair, all methods are evaluated on 10 random right-hand-side vectors. Table \ref{tab:1dIntegral} reports the average iteration count, evaluation time, and the standard deviation of the number of iterations.

In these experiments, our proposed method converges significantly faster than the two-grid method, requiring fewer iterations and less computational time. Besides, the iteration number of our proposed method is robust to variations in the regularization parameter $\alpha$, whereas the two-grid method is highly sensitive to changes in $\alpha$. Thanks to the highly parallelized structure of neural networks, the computational time of our method scales far more gradually with increasing problem size compared to the two-grid approach.
A graphical comparison (Figure \ref{fig:graph_comp_1d}) highlights the efficiency of our hybrid iterative method in resolving both low- and high-frequency components, while the two-grid method reduces only geometric low-frequency modes.

\begin{table}[]
\centering
\begin{tabular}{|>{\centering}p{0.12\textwidth}|>{\centering}p{0.08\textwidth}|>{\centering}p{0.2\textwidth}|>{\centering}p{0.2\textwidth}|>{\centering}p{0.2\textwidth}|}
\hline
   &    $\alpha$    &  $n=256$ & $n=512$ & $n=1024$ \tabularnewline \hline
        \multirow{3}{*}{\makecell{two-grid\\method}} & $10^{-3}$ & 2405(1.6s)$\pm 24$ & 2407(4s)$\pm 14$ & 2413(14s)$\pm 13$ \tabularnewline \cline{2-5} 
                  & $10^{-4}$ & 15928(11s)$\pm 233$ & 15951(27s)$\pm 145$ & 15983(91s)$\pm 123$\tabularnewline \cline{2-5} 
                                                  & $10^{-5}$ & 45705(26s)$\pm 904$ & 45808(76s)$\pm 550$ & 46000(311s)$\pm 478$
 \tabularnewline \hline
    \multirow{3}{*}{\makecell{proposed\\method}}    & $10^{-3}$ & 6.6(0.05s)$\pm 0.4$ & 7.0(0.06s)$\pm0.0$   & 6.8(0.08s) $\pm 0.4$\tabularnewline \cline{2-5} 
                 & $10^{-4}$ & 6.8(0.05s)$\pm0.4$ & 6.8(0.06s)$\pm0.1$ & 7.0(0.08s)$\pm0.5$ \tabularnewline \cline{2-5} 
                                                 & $10^{-5}$ & 6.8(0.05s)$\pm0.5$ & 7.0(0.06s)$\pm0.0$ & 6.9(0.08s)$\pm0.3$ \tabularnewline \hline
\end{tabular}
\caption{The average number of iterations and evaluation time required for solving the one-dimensional integral equation with Tikhonov regularization.}
\label{tab:1dIntegral}
\end{table}

\begin{figure}[]
    \centering
    \includegraphics[width=1.\linewidth]{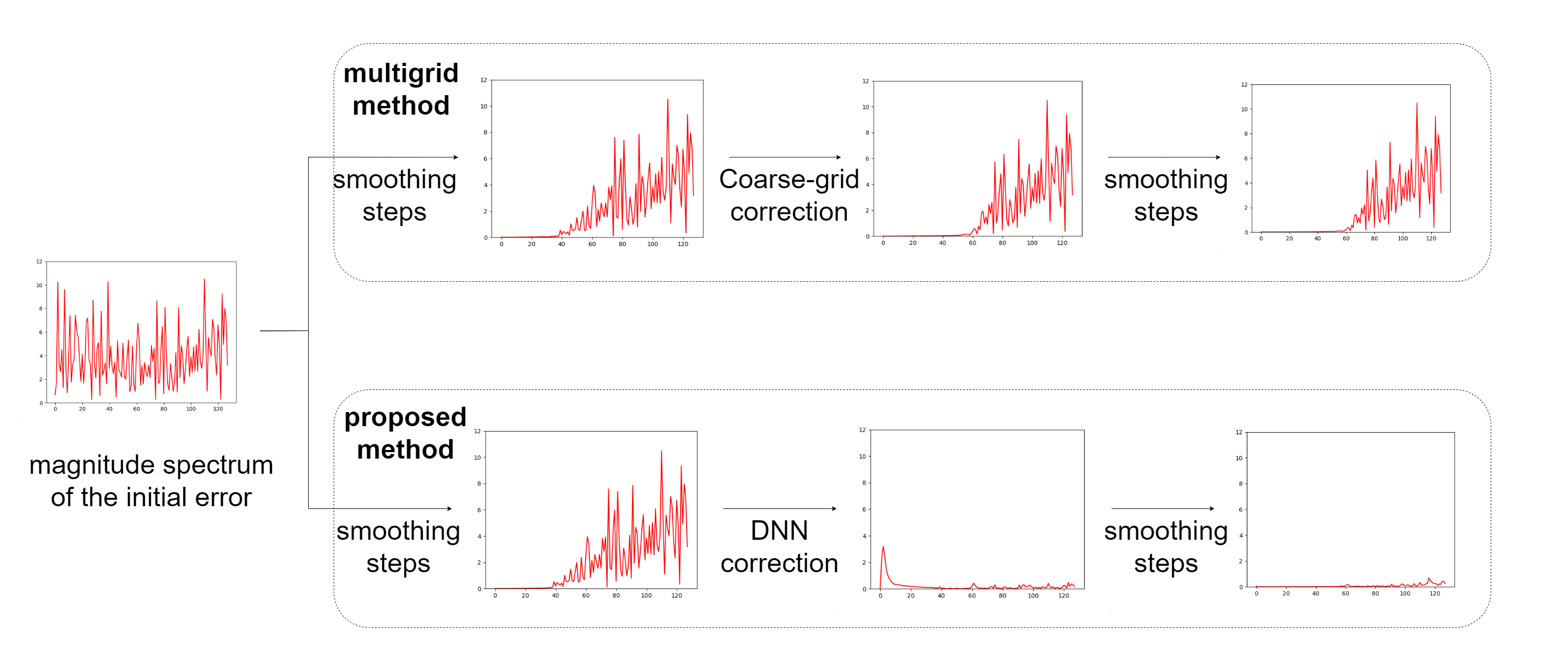}
    \caption{The magnitude spectra from our proposed method and the two-grid method. Since the magnitude spectra of real vectors are symmetric, here we show half of the spectra. In these figures, the x-axis represents the frequency domain where the left part corresponds to geometric low-frequency and the right part corresponds to geometric high-frequency.}
    \label{fig:graph_comp_1d}
\end{figure}
\subsubsection{\bf anisotropic diffusion regularized integral equation} \label{sec:exp_1d_ani}
Here, we also consider another regularization term 
$R(u)(z)=-\nabla\cdot(a(z)u(z))$
for the integral equation (\ref{eq:integral equation}), where $a(z)=1+0.5\sin(2\pi z)$. The coefficient matrix of the discretized system is $A_n=\alpha D_n+K_n$, where $D_n$ corresponds to the regularization term and $K_n$ corresponds to the convolution term. Here, $K_n$ is the same as in the previous subsection, and $D_n$ is a tri-diagonal matrix with varying coefficients along the diagonals. During the training, we choose the $P_n$ in the loss function (\ref{eq:expected_loss}) to be the inverse of the optimal circulant preconditioner \cite{chan1988optimal}. In this case, though $A_n$ is not a Toeplitz matrix, the optimal circulant preconditioner is a good approximation to $A_n$ if $a(z)$ is uniformly bounded away from zero and from above \cite{chan1990circulant}. The inverse of the circulant preconditioner can be efficiently implemented using the Fast Fourier transform. We keep the same implementation details as in the Tikhonov case and only change the weight of the weighted Jacobi iteration to $0.025$ for both the two-grid method and our proposed method. We compare our methods with the two-grid method for this problem with different $n$ and $\alpha$. The condition number of each case is listed in Table \ref{tab:condNum}, and the results are reported in Table \ref{tab:1dIntegral_2}. Our proposed algorithm consistently outperformed the two-grid method in both the iteration number and computational time. The convergence rate of our method is also robust to the change of $\alpha$ compared to the two-grid method.

\begin{table}[]
\centering
\begin{tabular}{|>{\centering}p{0.12\textwidth}|>{\centering}p{0.08\textwidth}|>{\centering}p{0.2\textwidth}|>{\centering}p{0.2\textwidth}|>{\centering}p{0.2\textwidth}|}
\hline
   &   $\alpha$     &  $n=256$ & $n=512$ & $n=1024$ \tabularnewline \hline
  \multirow{3}{*}{\makecell{two-grid\\method}}       & $10^{-3}$ & 1196(0.8s)$\pm28$ & 1200(2s)$\pm19$ & 1201(7s)$\pm11$ \tabularnewline \cline{2-5} 
                                                  & $10^{-4}$ & 8550(6s)$\pm288$ & 8640(15s)$\pm202$ & 8678(49s)$\pm117$\tabularnewline \cline{2-5} 
                                                  & $10^{-5}$ & 33876(22s)$\pm1724$ & 34338(60s)$\pm1214$ &  34566(196s)$\pm700$
 \tabularnewline \hline

 \multirow{3}{*}{\makecell{proposed\\method}}       & $10^{-3}$ & 5.1(0.05s)$\pm0.3$ & 5.6(0.06s)$\pm0.5$   & 6.0(0.07s)$\pm0.0$ \tabularnewline \cline{2-5} 
                                                 & $10^{-4}$ & 6.0(0.05s)$\pm0.0$ & 6.0(0.06s)$\pm0.0$   & 6.0(0.07s)$\pm0.0$ \tabularnewline \cline{2-5} 
                                                 & $10^{-5}$ & 6.0(0.05s)$\pm0.0$ & 6.1(0.06s)$\pm0.3$   & 7.1(0.08s)$\pm0.5$ \tabularnewline \hline
\end{tabular}
\caption{The average number of iterations and evaluation time required for solving the one-dimensional integral equation with anisotropic diffusion regularization.}
\label{tab:1dIntegral_2}
\end{table}

\subsection{Two-dimensional integral equations} \label{sec:exp_2d}
\subsubsection*{\bf Tikhonov regularized integral equation}
We also consider a two-dimensional Integral equation with Tikhonov regularization: $A_{n^2}=\alpha I_{n^2}+K_{n^2}$
where $I_{n^2}=I_n\otimes I_n$, $K_{n^2}=K_n\otimes K_n$, and $\otimes$ denotes the Kronecker product. The size of the system $A_{n^2}$ is $n^2\times n^2$. The $K_n$ here is the same as Section \ref{sec:exp_1d}. The matrix vector multiplications $A_{n^2}\mathbf x$ can be computed efficiently by using this tensor strucutre:
\begin{equation}
    A_{n^2}\mathbf x=\alpha \mathbf x+\text{Vec}(K_n^\top XK_n) \label{eq:An2x}
\end{equation}
where $X\in\mathbb R^{n\times n}$ is the matricization reshaping of $\mathbf x$, and Vec denotes the vectorization function, which is the inverse of matricization, of the input matrix. The multiplication $K_n^\top XK_n$ can further be implemented through the one-dimensional convolution.
{In the loss function $\mathcal{L}_{\rm{emp}}^N(\theta)$, we choose $P_{n^2}=aI_{n^2}-A_{n^2}$} with $a=1+\alpha$. We consider different choices of $\alpha$ and $n$, and the condition numbers of the corresponding systems are shown in Table \ref{tab:condNum}. The \textbf{training time} for our neural operators is about \textbf{$36$min}, \textbf{$152$min}, and \textbf{$645$min} when $n=256$, $512$, and $1024$ respectively.

To solve this linear system, we apply Algorithm \ref{alg:hybrid_iterative} with $\kappa_1=20$, $\kappa=0$, and the pre-smoothing step is performed by the weighted Jacobi iteration with weight 0.01. 
We compare our method against two benchmark approaches: the multigrid method and the preconditioned conjugate gradient (PCG) method, where the preconditioner for the PCG iteration is the block-circulant preconditioner from \cite{chan1992family}. During the implementations of all methods, we implement the matrix vector multiplications via (\ref{eq:An2x}). For the multigrid method, we do 20 steps of the weighted Jacobi iteration for pre-smoothing at each level, and do not perform post-smoothing (This is also called the $\backslash$-cycle). For $n=256$, 512, and 1024, the total multigrid level is chosen to be 5, 6, and 7, respectively. Similar to the one-dimensional cases, here the convergence of multigrid methods depends mainly on the number of smoothing steps on the finest grid rather than the specific multigrid scheme, because the main components in the error are only solvable on the finest grid.

The stopping criterion is when the relative residual error is below $10^{-10}$. Table \ref{tab:2dIntegral} presents the average number of iterations and evaluation time over 10 trials with randomly generated right-hand-side vectors. Our method demonstrates significant improvements over both the multigrid and PCG methods in terms of iteration number and computational efficiency.
Notably, while the PCG method relies on a carefully constructed preconditioner, our approach does not necessarily require one, which makes our method more broadly applicable. 
We also visualize the magnitude spectrum of the errors during iterations for our proposed method and the multigrid method in Figure \ref{fig:graph_comp_2d}. We observe that the multigrid method fails to address geometric high-frequency components, while ours can solve all frequency components efficiently. Besides, the cost of our methods increases very mildly when the matrix size and condition number increase.

\begin{table}[]
\centering
\begin{tabular}{|>{\centering}p{0.12\textwidth}|>{\centering}p{0.08\textwidth}|>{\centering}p{0.2\textwidth}|>{\centering}p{0.2\textwidth}|>{\centering}p{0.21\textwidth}|}
\hline
   &   $\alpha$     &  $n=256$ & $n=512$ & $n=1024$ \tabularnewline \hline
\multirow{3}{*}{\makecell{multigrid\\method}}       & $10^{-3}$ & 608(1.2min)$\pm3$  &  611(4.9min)$\pm 2$   & 611(21.1min)$\pm1$  \tabularnewline \cline{2-5} 
                                                 & $10^{-4}$ & 5072(10min)$\pm 34$ & 5096(41min)$\pm 18$ & 6008 (3hr)$\pm10$\tabularnewline \cline{2-5} 
                                                 & $10^{-5}$  & 41207(2hr)$\pm407$ & 41220(8hr)$\pm 211$ & 41228(33hr)$\pm 110$\tabularnewline \hline 
\multirow{3}{*}{\makecell{PCG\\method}}       & $10^{-3}$ & 89(1.3s)$\pm0.7$  &  94.5(7.2s)$\pm0.5$   & 98(26.3s)$\pm0.2$\tabularnewline \cline{2-5} 
                                                 & $10^{-4}$ & 162.5(2.4s)$\pm0.5$ & 173.0(11.9s)$\pm0.4$ & 180(46.7s)$\pm0.2$ \tabularnewline \cline{2-5} 
                                                 & $10^{-5}$ & 276(4.0s)$\pm2.4$ & 294(20.2s)$\pm1.9$  & 307(73.6s)$\pm1.0$ \tabularnewline \hline  
\multirow{3}{*}{\makecell{proposed\\method}}       & $10^{-3}$ & 8.6(0.19s)$\pm1.0$ & 9.8(1.36s)$\pm0.9$ &  9.8(5.10s)$\pm0.4$\tabularnewline \cline{2-5} 
                                                 & $10^{-4}$ & 11.2(0.26s)$\pm2.2$ & 12.1(1.69s)$\pm1.5$ & 12.3(6.5s)$\pm1.4$ \tabularnewline \cline{2-5} 
                                                 & $10^{-5}$ & 13.8(0.30s)$\pm0.7$ & 15.6(2.10s)$\pm1.8$ & 16.7(8.7s)$\pm1.8$ \tabularnewline \hline
\end{tabular}
\caption{The averaged number of iterations and evaluation time required for solving the two-dimensional integral equation.}
\label{tab:2dIntegral}
\end{table}

\begin{figure}[]
    \centering
    \includegraphics[width=.9\linewidth]{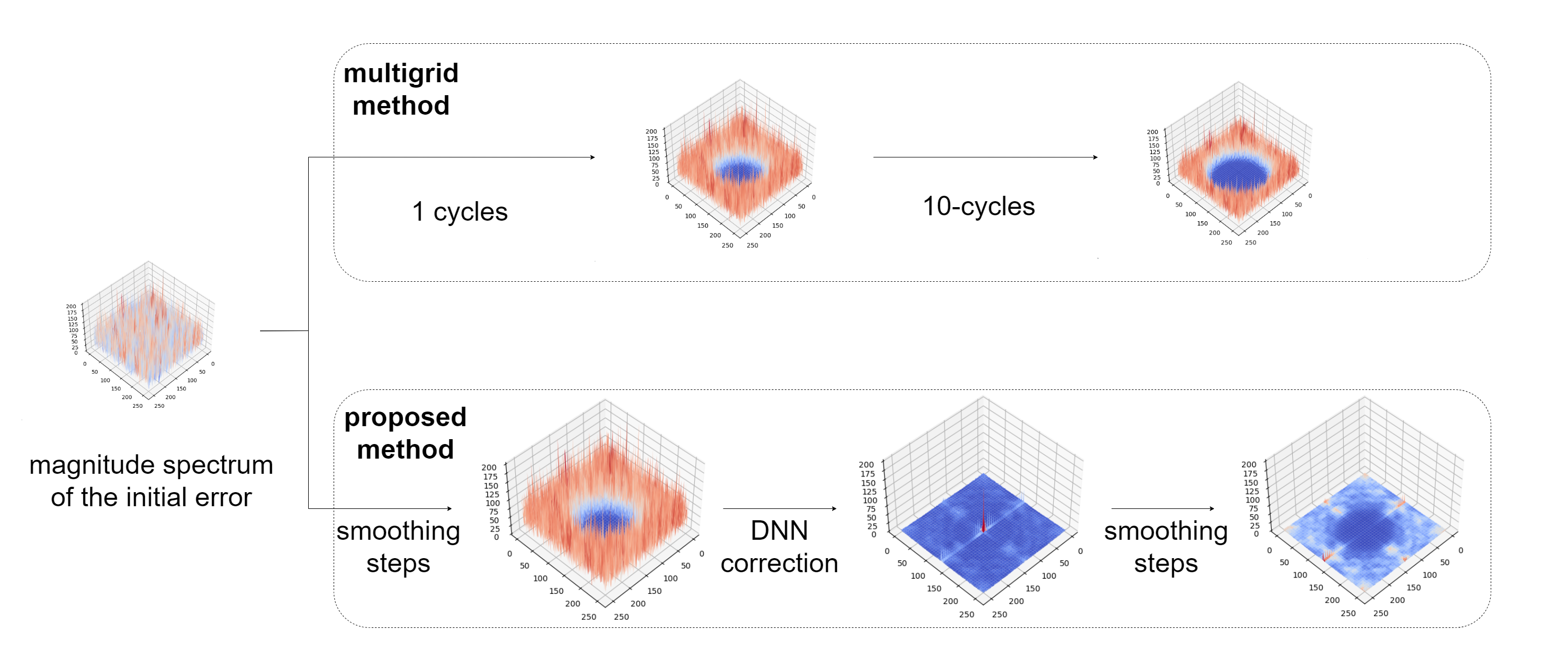}
    \caption{The magnitude spectra from the proposed method and the multigrid method. In these figures, the x-y plane represents the 2-dimensional shifted frequency domain where low-frequency components are shifted to the center, and high-frequency components are located on the corners and boundaries.}
    \label{fig:graph_comp_2d}
\end{figure}

\subsection{Transferability of neural operators}
In previous experiments, we independently trained distinct neural operators for each problem size $n$ when solving the same equation. While this approach suffices for solving systems of the form $A_n\mathbf x=\mathbf y$ with fixed $A_n$ and varying right-hand side y, it becomes computationally expensive when we need to solve an equation discretized at different $n$, because we need to retrain a new $\mathcal N_\theta$ for each $n$. Consequently, we investigate whether a neural operator trained for a small system can be easily generalized to larger systems discretized from the same equation. Thanks to the resolution-invariant structure of the FNO \cite{li2021fourier}, $\mathcal N_\theta$ can directly take input vectors of different sizes.
We consider the one-dimensional integral equation (\ref{eq:discretized_system_tik_1d}). For each value of $\alpha=10^{-3}$, $10^{-4}$, and $10^{-5}$, we first apply the neural operator $\mathcal N_{\theta}$ trained on $n=256$ directly to systems with $n=512$ and $n=1024$ without any modification. Then, we fine-tune $\mathcal N_{\theta}$ on different $n$, \text{i.e.}, continue to train $\mathcal N_{\theta}$ for 50 epochs on the new system, to see the improvements. The numerical results are presented in Table \ref{tab:transferability}. We observe that directly applying $\mathcal{N}_\theta$ trained on small $n$ to larger systems can still work, but it takes much more iterations to converge. However, the convergence performance improves significantly after only a few epochs of fine-tuning. This suggests that our neural operator retains transferability across different system sizes, requiring only minimal fine-tuning effort for effective adaptation.  
\begin{table}[]
\centering
\begin{tabular}{|>{\centering}p{0.15\textwidth}|>{\centering}p{0.1\textwidth}|>{\centering}p{0.25\textwidth}|>{\centering}p{0.25\textwidth}|}
\hline
   &   $\alpha$     &  Directly apply & After fine-tunning \tabularnewline \hline
\multirow{3}{*}{$n=512$}       & $10^{-3}$ & 38.3(0.33s)$\pm5.2$ & 9.7(0.08s)$\pm1.0$   \tabularnewline \cline{2-4} 
                                                 & $10^{-4}$ & 63.7(0.56s)$\pm4.4$ & 9.9(0.08s)$\pm0.8$   \tabularnewline \cline{2-4} 
                                                 & $10^{-5}$ & 47.1(0.42s)$\pm2.5$ & 8.8(0.08s)$\pm0.6$ \tabularnewline \hline

  \multirow{3}{*}{$n=1024$}       & $10^{-3}$ & 51.1(0.60s)$\pm3.6$ & 12.1(0.15s)$\pm1.6$  \tabularnewline \cline{2-4} 
                                                  & $10^{-4}$ & 113.6(1.32s)$\pm6.9$ & 14.2(0.17s)$\pm4.4$ \tabularnewline \cline{2-4} 
                                                  & $10^{-5}$ & 79.0(0.92s)$\pm 10.4$ & 10.8(0.13s)$\pm0.7$
 \tabularnewline \hline
\end{tabular}
\caption{The average number of iterations and evaluation time when transferring models trained with $n=256$ to $n=512$ and $n=1024$.}
\label{tab:transferability}
\end{table}

\subsection{Validating the convergence rate of the generalization error}
In Theorem \ref{theorem:generalization_error}, we derive an analytical upper bound for training neural operators. By fixing the structure of neural operators, the expected $L_2$ error would converge at the order of $O(N^{-1/2})$ if we can find the global minimizer, \textit{i.e.}, $\delta=0$. Here, we numerically estimate the convergence rate and compare it with the theoretical result.

We consider solving the system $A_n=\alpha I_n+K_n$ (the same as Section \ref{sec:exp_1d_tik}) with $n=32$ and $\alpha=0.01$. Let $\mathcal{N}_\theta$ be a two-layer network defined as (\ref{eq:two_layer_network}). We then randomly generate a finite set of training samples and train $\mathcal{N}_\theta$ by minimizing the empirical loss (\ref{eq:empirical_loss}). Since our generalization error estimate in Theorem \ref{theorem:generalization_error} works for general $P_n$, we set $P_n=I_n$ in the loss function for simplicity. Due to the existence of the optimization error $\delta$, the actual convergence rate might be lower than the theoretical rate. To mitigate the effect of optimization error, we train $\mathcal{N}_\theta$ with 5 different random initializations and select the model with the minimum empirical loss. The testing error is shown in Table \ref{tab:convergence_rate}. Here, we choose the number of training samples $N_i=10\times2^i$ for $i=0,1,\dots,4$ and computed the generalization error accordingly. We also estimated the convergence rate for $i=1,2,3,4$ by 
\[ \log\left(\frac{\text{error}(N_{i-1})}{\text{error}(N_{i})}\right)\Big/\log\left(\frac{N_{i-1}}{N_{i}}\right), \]
where $\text{error}(N_i)$ represents the generalization error when training on $N_i$ samples, and shows the convergence rate (marked in parentheses for each $N_i$) as well. We see the convergence rate is close to the theoretical result, \textit{i.e.}, $-0.5$, when $N$ is small. When $N$ is large, the optimization error will start to dominate the generalization error, so the observed convergence rate also decreases.

\begin{table}[]
\centering
\begin{tabular}{|>{\centering}p{0.11\textwidth}|>{\centering}p{0.13\textwidth}|>{\centering}p{0.13\textwidth}|>{\centering}p{0.13\textwidth}|>{\centering}p{0.13\textwidth}|>{\centering}p{0.13\textwidth}|}
\hline
 $N_i$ &  $N_1=10$   & $N_2=20$ & $N_3=40$ & $N_4=80$ & $N_5=160$   \tabularnewline \hline
error &  0.92   &  0.67(-0.46)   & 0.47(-0.51) &  0.32(-0.55) & 0.26(-0.30)  \tabularnewline \hline
\end{tabular}
\caption{Generalization error and estimated convergence rate (marked inside parentheses) with respect to different numbers of training samples.}
\label{tab:convergence_rate}
\end{table}


\subsection{Validation of Assumption 1 in Theorem \ref{theorem:error_estimate}}
We also validate the correcting condition assumed in Theorem \ref{theorem:error_estimate}. We consider the one-dimensional and two-dimensional integral equation with Tikhonov regularization, and calculate the ratio
\begin{equation} \frac{\Vert \mathcal{N}(\mathbf r)-A^{-1}_n\mathbf r\Vert^2_{A_n}}{\Vert \mathcal{N}(\mathbf r)-A^{-1}_n\mathbf r\Vert^2_{A_n^2} /a_0}\label{eq:correcting_condition2}\end{equation}
for 100 randomly generated right-hand-sides $\mathbf r$. The results are shown in Table \ref{tab:assumption2_ratio}. In all scenarios, the ratios (\ref{eq:correcting_condition2}) are very stable, indicating that the correcting condition holds with high probability for some $\delta_2$. 
\begin{table}[h]
\centering
\begin{tabular}{|>{\centering}p{0.08\linewidth}|>{\centering}p{0.08\linewidth}|>{\centering}p{0.22\linewidth}|>{\centering}p{0.22\linewidth}|>{\centering}p{0.22\linewidth}|}
\hline
   &    $\alpha$    &  $n=256$ & $n=512$ & $n=1024$ \tabularnewline \hline
\multirow{3}{*}{\makecell{$d=1$}}       & $10^{-3}$ & $0.267\pm2.97\times10^{-5}$ & $0.267\pm2.91\times10^{-5}$   & $0.267\pm2.87\times10^{-7}$ \tabularnewline \cline{2-5} 
                                                 & $10^{-4}$ & $0.266\pm4.27\times10^{-6}$ & $0.266\pm2.61\times10^{-7}$ & $0.266\pm2.64\times10^{-8}$ \tabularnewline \cline{2-5} 
                                                 & $10^{-5}$ & $0.266\pm4.99\times10^{-6}$ & $0.266\pm7.73\times10^{-6}$ & $0.266\pm6.39\times10^{-5}$ \tabularnewline \hline
\multirow{3}{*}{$d=2$}       & $10^{-3}$ & $0.290\pm1.71\times10^{-3}$ & $0.304\pm2.87\times10^{-3}$   & $0.493\pm9.84\times10^{-3}$ \tabularnewline \cline{2-5} 
                                                 & $10^{-4}$ & $0.270\pm1.86\times10^{-4}$ & $0.266\pm1.51\times10^{-5}$ & $0.266\pm1.16\times10^{-6}$ \tabularnewline \cline{2-5} 
                                                 & $10^{-5}$ & $0.266\pm1.93\times10^{-6}$ & $0.266\pm9.97\times10^{-6}$ & $0.279\pm3.01\times10^{-4}$ \tabularnewline \hline
\end{tabular}
\caption{Validation of the correcting condition}
\label{tab:assumption2_ratio}
\end{table}

\section{Conclusion and future works}
In this work, we combine classical iterative solvers and neural operators to solve large-scale linear systems arising from convolution-type integral equations. We propose a new training strategy for neural operators that can address algebraic low-frequency components efficiently. We also estimate the generalization error for our proposed training strategy with two-layer neural networks, providing an upper bound on the generalization error in terms of the number of training samples, problem size, and the optimization error. In the future, we will also try to estimate the generalization error for some general multi-layer networks like \cite{ohn2019smooth}.

In the numerical experiments, our algorithm outperforms some classical numerical methods, such as the multigrid and the PCG method, by a huge margin in both iteration numbers and computational time. Our method has demonstrated the ability to address large-scale and ill-conditioned linear systems. Thanks to the highly parallel structures of neural networks, the computational time increases gradually when the problem size increases. We also show that the proposed method can be easily transferred to systems discretized at different sizes with mild cost.

Though we focus on a certain type of integral equations (the integral equation of the second kind), the proposed algorithm and training strategy can be generalized to other types of problems directly. We will consider applying the proposed method to different integral equations, \textit{e.g.}, the integral equation of the first kind, in the following works.
One limitation of our proposed method is the requirement for training, which takes a much longer time than evaluation. When the underlying system changes frequently, we need to retrain our neural operators many times. {To address this issue, we will try to design multiple-input neural operators, like the Meta-MgNet \cite{chen2022meta}, that can solve a class of linear systems instead of one specific system.}


\bibliography{bibfile} 
\bibliographystyle{siamplain}

\appendix
\gdef\thesection{\Alph{section}} 
\makeatletter
\renewcommand\@seccntformat[1]{\appendixname\ \csname the#1\endcsname.\hspace{0.5em}}
\makeatother
\section{Rademacher complexity and generalization error}{\label{sec:rc_complexity}}
In this section, we introduce the definition of the Rademacher complexity and how to estimate the generalization error using this complexity. 
\begin{definition}
    Let $\{\mathbf x_i\in\mathbb R^n\}_{i=1}^N$ be a set of i.i.d. random variables drawn from certain distribution $\mathcal{X}$ and $\{\varepsilon_i\}_{i=1}^N$ be a set of i.i.d Rademacher variables (\textit{i.e.}, only taking value of $1$ and $-1$ with equal probability). Then the \textbf{empirical Rademacher Complexity} of the function class $\mathcal{F}$ is defined as
    \begin{equation*}
      \hat{R}_N(\mathcal{F}):= \mathbb{E}_\varepsilon\left(\sup_{f\in\mathcal{F}}\left(\frac{1}{N}\sum_{i=1}^N \varepsilon_i f(\mathbf x_i)\right)\right).
    \end{equation*}
    Taking its expectation over $\mathcal X$ yields the \textbf{Rademacher Complexity} of the function class $\mathcal{F}$
    \begin{equation*}
      R_N(\mathcal{F}):= \mathbb{E}_\mathcal{X}\left(\hat{R}_N(\mathcal{F})\right)=\mathbb{E}_\mathcal{X}\mathbb{E}_\varepsilon\left(\sup_{f\in\mathcal{F}}\left(\frac{1}{N}\sum_{i=1}^N \varepsilon_i f(\mathbf x_i)\right)\right).
    \end{equation*} \label{def:Rad_comp}
\end{definition}
\begin{remarks}
    In the rest of this section, we assume that all $\mathbf x_i$ sampled from $\mathcal{X}$ are uniformly bounded: $\Vert \mathbf x_i\Vert_2\leq 1$.
\end{remarks}
Some calculation rules of the Rademecher complexity are given below
\begin{lemma}
    Let $\mathcal{F},\mathcal{G}$ be function classes and $a,b$ be constants. Then
    \begin{enumerate}
        \item $R_N(\mathcal F)\leq R_N(\mathcal G)$ if $\mathcal F\subseteq \mathcal G$
        \item $R_N(\mathcal{F+G})= R_N(\mathcal{F}) + R_N(\mathcal{G})$.
        \item $R_N(a\mathcal{F}) = |a|R_N(\mathcal{F}).$
        \item (contraction lemma \cite{ledoux2013probability}) Assume that $\sigma:\mathbb{R}\mapsto\mathbb{R}$ is a $L$-Lipschitz function, then $R_N(\sigma(\mathcal{F}))\leq LR_N(\mathcal{F})$.
        \item If $\Vert f\Vert_\infty$ is uniformly bounded for any $f\in\mathcal{F}$, then $R_N(\mathcal{F}^2)\leq 2\Vert f\Vert_\infty R_N(\mathcal{F})$ where $\mathcal{F}^2:=\{f(\mathbf x)^2|f\in\mathcal{F}\}$. 
        \item If $\Vert f\Vert_\infty$ is uniformly bounded for any $f\in\mathcal{F}$, then $R_N(\mathcal{F}\times\mathcal{F})\leq 6\Vert f\Vert_\infty R_N(\mathcal{F})$ where $\mathcal{F}\times\mathcal{F}:=\{f_1(\mathbf x)f_2(\mathbf x)|f_1,f_2\in\mathcal{F}\}$.
    \end{enumerate}
    \label{lem::computing_rules}
\end{lemma}
\begin{proof}
    The first 3 computation rules can be directly derived from the definition. 
    The proof of the contraction lemma can be found in \cite[Lemma 26.9]{Shalev-Shwartz_Ben-David_2014}. Rule 5 is a direct application of the contraction lemma. Though $\sigma(x)=(x)^2$ is not a Lipschitz function when $x\in\mathbb R$, its derivative $2x$ is bounded when $x$ is bounded. Since $\sup_{f\in\mathcal F}\Vert f\Vert_\infty$ is assumed to be finite, we can apply the contraction lemma to estimate $R_N(\mathcal F^2)$ directly. 
    What remains is to prove property 6. For any $f_1f_2$, we can write it as $f_1f_2=\frac{1}{2}(f_1+f_2)^2-\frac{f_1^2}{2}-\frac{f_2^2}{2}$, so we have $\mathcal F\times\mathcal F\subseteq \frac{1}{2}(\mathcal F+\mathcal F)^2-\frac{1}{2}\mathcal F^2-\frac{1}{2}\mathcal F^2$. Then, applying calculation rules 1 to 5, we have
    \begin{align*}
        &R_N(\mathcal{F}\times\mathcal{F})\leq R_N\left(\frac{1}{2}(\mathcal F+\mathcal F)^2-\frac{1}{2}\mathcal F^2-\frac{1}{2}\mathcal F^2\right)\leq R_N\left(\frac{1}{2}(\mathcal F+\mathcal F)^2\right)+R_N(\mathcal F^2)\\
        &\leq \frac{1}{2}2(\sup_{f\in\mathcal F}\Vert f\Vert_{\infty}+\sup_{f\in\mathcal F}\Vert f\Vert_{\infty})(2R_N(\mathcal F))+2\sup_{f\in\mathcal F}\Vert f\Vert_{\infty}R_N(\mathcal F)=6\sup_{f\in\mathcal F}\Vert f\Vert_{\infty}R_N(\mathcal F).
    \end{align*}
    The same technique is used in the proof of \cite[Lemma 4.4]{li2024priori}. 
\end{proof}

Lemma \ref{lem:linear_trans} (\cite[Lemma 26.10]{Shalev-Shwartz_Ben-David_2014}) gives an upper bound for the Rademacher complexity for a class of linear functions.
\begin{lemma}
    Let $\mathcal{G}$ be the linear transformation function class defined by 
    \begin{equation*}
        \mathcal{G}:= \{g(\mathbf x)=\mathbf w^\top \mathbf x|\mathbf w\in\mathbb{R}^n,\|\mathbf w\|_2\leq C\},
    \end{equation*}
    where $C>0$.
    Then we have $R_N(\mathcal{G})\leq \frac{\sup_{\mathbf x\sim\mathcal X}\Vert \mathbf x\Vert_2 C}{\sqrt{N}}.$
    \label{lem:linear_trans}
\end{lemma}
Using Lemma \ref{lem:linear_trans}, we can estimate the Rademacher complexity of the composition of a ReLU function and linear transformations.
\begin{lemma}
    Let $\mathcal{G}$ be defined by 
    \begin{equation*}
        \mathcal{G}:= \{g(\mathbf x)=a\textnormal{ReLU}(\mathbf w^\top \mathbf x)|\mathbf w\in\mathbb{R}^n,\|\mathbf w\|_2\leq 1, |a|\leq C\},
    \end{equation*}
    where $C>0$.
    Then we have $R_N(\mathcal{G})\leq \frac{\sup_{\mathbf x\sim\mathcal X}\Vert \mathbf x\Vert_2 C}{\sqrt{N}}.$
    \label{lem:relu_complexity}
\end{lemma}
\begin{proof}
    Let $\{\mathbf x_i\}_{i=1}^N$ be a set of randomly generated vectors.
    In Definition \ref{def:Rad_comp}, we have
    \begin{align*}
        &\sup_{g\in\mathcal G} \left( \frac{1}{N}\sum_{i=1}^N \epsilon_ig(\mathbf x_i) \right)=\sup_{a\in[-C,C],\Vert\mathbf w\Vert_2\leq 1} \left( \frac{a}{N}\sum_{i=1}^N \epsilon_i\textnormal{ReLU}(\mathbf w^\top \mathbf x_i)\right)\\
        &\leq C \sup_{\Vert\mathbf w\Vert_2\leq 1} \left( \frac{1}{N}\sum_{i=1}^N \epsilon_i\textnormal{ReLU}(\mathbf w^\top \mathbf x_i)\right),
    \end{align*}
    which implies that $R_N(\mathcal G)\leq CR_N(\{\textnormal{ReLU}(\mathbf w^\top \mathbf x)|\mathbf w\in\mathbb{R}^n,\|\mathbf w\|_2\leq 1\})$. Then, using the Lipschitz continuity of the ReLU function and the contraction lemma (Lemma \ref{lem::computing_rules}{(3)}), we get
    \[ R_N(\mathcal G)\leq CR_N(\{\mathbf w^\top \mathbf x|\mathbf w\in\mathbb{R}^n,\|\mathbf w\|_2\leq 1\})\leq\frac{C\sup_{\mathbf x}\Vert \mathbf x\Vert_2}{\sqrt{N}},\]
    where the last inequality is derived from Lemma \ref{lem:linear_trans}.
\end{proof}
Let $l:\mathbb{R}\rightarrow\mathbb{R}$ be a given continuous function and $\mathcal{L}:=\{l(f(\mathbf x))|f\in\mathcal{F}\}$ be a function class. In machine learning problems, $l$ usually represents the loss function that measures the difference between the model $f$ and a given ground truth label. For any $f\in\mathcal{F}$, we define
\begin{equation}
    L_\mathcal{X}(f):=\mathbb{E}_{\mathbf x\sim\mathcal{X}}(l(f(\mathbf x))) \text{ and } L_N(f):=\frac{1}{N}\sum_{i=1}^N(l(f(\mathbf x_i))).
\end{equation}
The difference $|L_N(f)-L_\mathcal{X}(f)|^2$ is called the generalization error, which can be bounded by the Rademacher complexity using Lemma \ref{lemma:quadrature_error} (\cite[Theorem 26.5]{Shalev-Shwartz_Ben-David_2014}).
\begin{lemma}\label{lemma:quadrature_error}
    Let $\mathcal{F}$ be a set of functions, $\{\mathbf x_1,\cdots, \mathbf x_N\}$ be i.i.d. random variables sampled from $\mathcal{X}$. Then, for any $\epsilon>0$ and $f\in\mathcal F$, the following inequality holds with probability at least $1-\epsilon$,
     \begin{equation}
         L_\mathcal{X}(f) - L_N(f)\leq 2R_N(l\circ \mathcal{F})+\sup_{l(f)\in l\circ\mathcal F}\Vert l(f)\Vert_\infty\sqrt{\frac{2\log(2/\epsilon)}{N}},
     \end{equation}
     where $l\circ \mathcal{F}:=\{l(f)|f\in\mathcal{F}\}$.
\end{lemma}

\section{Proofs}
\subsection{Proof of Theorem \ref{theorem:generalization_error}} \label{sec:generalization_error_proof}
\begin{proof}
    We first construct a set of parameters $\hat{\theta}\in\Theta(C_n)$ as
$$ \hat{\theta}=\left(\Vert A_n^{-1}\Vert_2,-\Vert A_n^{-1}\Vert_2,\frac{A_n^{-1}}{\Vert A_n^{-1}\Vert_2},-\frac{A_n^{-1}}{\Vert A_n^{-1}\Vert_2}\right).$$
Then, we have
$$ \mathcal{N}_{\hat{\theta}}(\mathbf y)=\Vert A_n^{-1}\Vert_2\text{ReLU}\left(\frac{A_n^{-1}}{\Vert A_n^{-1}\Vert_2}\mathbf y\right)-\Vert A_n^{-1}\Vert_2\text{ReLU}\left(\frac{-A_n^{-1}}{\Vert A_n^{-1}\Vert_2}\mathbf y\right)=A_n^{-1}\mathbf y $$
for any $\mathbf y\in\mathbb{R}^n$, which implies that $\mathcal{N}_{\hat{\theta}}$ exactly recover the operator $A_n^{-1}$ and $\mathcal{L}_{\rm{exp}}(\hat{\theta})=\mathcal{L}_{\rm{emp}}^N(\hat{\theta})=0$. 

The generalization errors of $\mathcal N_{\theta^*}$ is bounded as follows:
\begin{align*}
   &\Vert \mathcal{N}_{\theta^*}-A_n^{-1}\Vert^2_{A_n,\mu} = \mathbb{E}_{\mathbf y\sim\mu}(\Vert A_n^{1/2}(\mathcal{N}_{\theta^*}(\mathbf y)-A_n^{-1}\mathbf y)\Vert^2_2)\\
    &=\mathbb{E}_{\mathbf y\sim\mu}(\Vert A_n^{1/2}P_n^{-1}P_n(\mathcal{N}_{\theta^*}(\mathbf y)-A_n^{-1}\mathbf y)\Vert^2_2)\\
    &\leq \Vert A_nP_n^{-2}\Vert_2(\mathcal{L}_{\rm{exp}}(\theta^*)-\mathcal{L}_{\rm{emp}}^N(\theta^*)+\mathcal{L}_{\rm{emp}}^N(\theta^*))\\
    &\leq \Vert A_nP_n^{-2}\Vert_2(\mathcal{L}_{\rm{exp}}(\theta^*)-\mathcal{L}_{\rm{emp}}^N(\theta^*)+\mathcal{L}_{\rm{emp}}^N(\hat{\theta})+\delta)
\end{align*}
Here, $\mathcal{L}_{\rm{emp}}^N(\hat{\theta})=0$ and $\mathcal{L}_{\rm{exp}}(\theta^*)-\mathcal{L}_{\rm{emp}}^N(\theta^*)$ is bounded by Lemma \ref{lemma:quadrature_error}. We define the function class $\mathcal{H}:=\{h(\mathbf y)=\Vert P_n(\mathcal{N}_\theta(\mathbf y)-A_n^{-1}\mathbf y)\Vert^2_2|\theta\in\Theta(C_n)\}$. By Lemma \ref{lemma:quadrature_error}, the error $\mathcal{L}_{\rm{exp}}(\theta^*)-\mathcal{L}_{\rm{emp}}^N(\theta^*)$ is bounded by $2R_N(\mathcal{H})+\sup_{h\in\mathcal{H}}\Vert h\Vert_\infty\sqrt{\frac{2\log(2/\epsilon)}{N}}$ with probability at least $1-\epsilon$. Consequently, the generalization error is bounded, with probability at least $1-\epsilon$, by
\begin{equation}
    \Vert \mathcal{N}_{\theta^*}-A_n^{-1}\mathbf y\Vert^2_{A_n,\mu} \leq \Vert A_nP_n^{-2}\Vert_2(2R_N(\mathcal{H})+\sup_{h\in\mathcal{H}}\Vert h\Vert_\infty\sqrt{\frac{2\log(2/\epsilon)}{N}}+\delta). \label{eq:generalization_error_bound}
\end{equation} 
For every $h\in\mathcal{H}$ and $\mathbf y\in\mathbb{B}(\mathbb{R}^d)$, 
$$|h(\mathbf y)|\leq\Vert P_n\Vert_2^2\Vert c_1\text{ReLU}(W_1\mathbf y)+c_1\text{ReLU}(W_2\mathbf y)-A_n^{-1}\mathbf y\Vert^2_2\leq 9\Vert P_n\Vert_2^2C_n^2.$$
What remains is bounding the complexity $R_N(\mathcal{H})$. Notice that
\begin{align}
    &\Vert P_n(\mathcal{N}_\theta(y)-A_n^{-1}\mathbf y)\Vert^2_2
    =\sum_{j=1}^n \left(\sum_{k=1}^nP_n[j,k](\mathcal{N}(\mathbf y)_\theta-A_n^{-1}\mathbf y)[k]\right)^2 \notag\\
    &=\sum_{j=1}^n \left(\sum_{k=1}^n\sum_{l=1}^n P_n[j,k]P_n[j,l](\mathcal{N}_\theta(\mathbf y)-A_n^{-1}\mathbf y)[k](\mathcal{N}_\theta(\mathbf y)-A_n^{-1}\mathbf y)[l]\right).\label{eq:rad_comp_est}
\end{align}
Here, $[\cdot]$ denotes the indices of a matrix or vector. We define another class of functions $\mathcal{G}$ as 
\begin{align*}
    \bigg\{g(\mathbf y)=c_1\text{ReLU}(\mathbf w_1^\top \mathbf y)+c_2\text{ReLU}(\mathbf w_2^\top \mathbf y)-\mathbf w_3^\top \mathbf y\bigg||c_j|\leq C_n, \Vert \mathbf w_j\Vert_2\leq 1,j=1,2;\Vert\mathbf w_3\Vert_2\leq \Vert A_n^{-1}\Vert_2\bigg\}.    
\end{align*} 
Then, each $(\mathcal{N}_\theta(\mathbf y)-A_n^{-1}\mathbf y)[k]$, $k=1,\dots,n$, belongs to $\mathcal{G}$, and by (\ref{eq:rad_comp_est}), 
\begin{equation*}
    \mathcal H\subset \sum_{j=1}^n \left(\sum_{k=1}^n\sum_{l=1}^n P_n[j,k]P_n[j,l](\mathcal G\times\mathcal G)\right).
\end{equation*}
For any $g\in\mathcal{G}$ and $\mathbf y\sim\mu$, we have $|g(\mathbf y)|\leq 3C_n$. Then, using Lemma \ref{lem::computing_rules},
\begin{align}
R_N(\mathcal{H})&\leq\sum_{j=1}^n \left(\sum_{k=1}^n\sum_{l=1}^n P_n[j,k]P_n[j,l]R_N(\mathcal G\times\mathcal G)\right) & (\text{By Lemma \ref{lem::computing_rules}(1-3)}) \notag\\
&\leq\sum_{j=1}^n \left(\sum_{k=1}^n\sum_{l=1}^n \left(\frac{1}{2}P_n[j,k]^2+\frac{1}{2}P_n[j,l]^2\right)R_N(\mathcal G\times\mathcal G)\right) & \notag\\
&\leq n\Vert P_n\Vert_{F}^2(R_N(\mathcal{G}\times\mathcal{G})) & \notag\\
& \leq 18nC_n\Vert P_n\Vert_{F}^2R_N(\mathcal{G}) & (\text{By Lemma \ref{lem::computing_rules}(6)}) \label{eq:H_decomposition}
\end{align}
We further estimate $R_N(\mathcal{G})$ by $R_N(\mathcal{G})\leq 2R_N(\mathcal{G}_1)+R_N(\mathcal{G}_2)$.
Here $\mathcal{G}_1$ is defined as $\{a\text{ReLU}(\mathbf w^\top \mathbf y)|\Vert \mathbf w\Vert_2\leq 1,|a|\leq C_n\}$ and $\mathcal{G}_2$ is defined as $\{\mathbf w^\top \mathbf y|\Vert \mathbf w\Vert_2\leq C_n \}$. 
Then, we directly apply Lemma \ref{lem:linear_trans} and Lemma \ref{lem:relu_complexity} to the above inequality, and get
$$ R_N(\mathcal{G})\leq 3C_n\frac{\sup_{\mathbf y\sim\mu} \Vert \mathbf y\Vert_2}{\sqrt{N}}\leq \frac{3C_n}{\sqrt{N}}.$$
Then, by applying the above estimate of $R_N(\mathcal G)$ to (\ref{eq:H_decomposition}), we have
\begin{equation}
    R_N(\mathcal{H})\leq 54nC_n^2\Vert P_n\Vert_{F}^2/\sqrt{N}. \label{eq:RH_estimate}
\end{equation}
Finally, by combining (\ref{eq:RH_estimate}) and (\ref{eq:generalization_error_bound}), we have
\begin{align*}
    \Vert \mathcal{N}_{\theta^*}-A_n^{-1}\Vert^2_{A_n,\mu} \leq &\Vert A_nP_n^{-2}\Vert_2\left(2R_N(\mathcal{H})+\sup_{h\in\mathcal{H}}\Vert h\Vert_\infty\sqrt{\frac{2\log(2/\epsilon)}{N}}+\delta\right)\\
    \leq &\Vert A_nP_n^{-2}\Vert_2\left( \frac{108nC_n^2\Vert P_n\Vert_{F}^2}{\sqrt{N}}+9C_n^2\Vert P_n\Vert_2^2\sqrt{\frac{2\log(2/\epsilon)}{N}}+\delta\right)
\end{align*}
with probability at least $1-\epsilon$.

\end{proof}

\subsection{proof of Theorem \ref{theorem:error_estimate}} \label{sec:proof_error_estimate}
\begin{proof}
    Following Algorithm \ref{alg:hybrid_iterative}, $\mathbf x_{\tau+1}$ is calculated as 
    $$ 
        \mathbf r=\frac{\mathbf y-A_n\mathbf x_\tau}{\Vert \mathbf y-A_n\mathbf x_\tau\Vert_2};\quad
        \hat{\mathbf x}_\tau=\mathcal{N}_\theta(\mathbf r)\Vert \mathbf y-A_n\mathbf x_\tau\Vert_2+\mathbf x_\tau;\quad
        \mathbf x_{\tau+1}=\mathcal{S}(A_n,\mathbf y,\hat{\mathbf x}_\tau,1)
        $$
    Then, we have
    \begin{align*}
    &\Vert \mathbf x_{\tau+1}-A_n^{-1}\mathbf y\Vert^2_{A_n}=\Vert S_n(\hat{\mathbf x}_{\tau}-A_n^{-1}\mathbf y) \Vert^2_{A_n}\\
    &\leq \Vert \hat{\mathbf x}_{\tau}-A_n^{-1}\mathbf y\Vert^2_{A_n}-\frac{\delta_1}{a_0}\Vert \hat{\mathbf x}_{\tau}-A_n^{-1}\mathbf y\Vert^2_{A_n^2}\\
    &=\Vert \mathbf y-A_n\mathbf x_\tau\Vert_2^2\left(\Vert \mathcal{N}_\theta(\mathbf r)-A_n^{-1}\mathbf r\Vert^2_{A_n}-\frac{\delta_1}{a_0}\Vert \mathcal{N}_\theta(\mathbf r)-A_n^{-1}\mathbf r\Vert^2_{A_n^2}\right)\\
    &\leq (1-\frac{\delta_1}{\delta_2})\Vert\mathbf  y-A_n\mathbf x_\tau\Vert_2^2\Vert \mathcal{N}_\theta(\mathbf r)-A_n^{-1}\mathbf r\Vert^2_{A_n}\\
    &\leq (1-\frac{\delta_1}{\delta_2})\delta_3\Vert \mathbf y-A_n\mathbf x_\tau\Vert_2^2\Vert A_n^{-1}\mathbf r\Vert^2_{A_n}=(1-\frac{\delta_1}{\delta_2})\delta_3\Vert \mathbf x_\tau-A_n^{-1}\mathbf y\Vert^2_{A_n}
\end{align*}
\end{proof}

\subsection{Proof of Corollary \ref{cor_stability}} \label{sec:proof_stability}
\begin{proof}
Since $\Vert \mathbf r\Vert_2=1$, we have
$$\Vert A_n^{-1}\mathbf r\Vert_{A_n}\geq (\lambda_{\min}(A_n^{-1})\Vert \mathbf r\Vert_2^2)^{1/2}= \frac{1}{\lambda_{\max}(A_n)^{1/2}}=\frac{1}{\Vert A_n\Vert_2^{1/2}}.$$
Then, 
\begin{align}
    &\text{Prob}\left(\Vert \mathcal{N}_{\theta^*}(\mathbf r)-A_n^{-1}\mathbf r \Vert^2_{A_n}\geq \delta_3\Vert A_n^{-1}\mathbf r\Vert^2_{A_n}\right)\notag\\
    &\leq\text{Prob}\left(\Vert \mathcal{N}_{\theta^*}(\mathbf r)-A_n^{-1}\mathbf r \Vert^2_{A_n}\geq \frac{\delta_3}{\Vert A_n\Vert_2}\right)\notag\\
    \leq & \frac{\Vert A_nP_n^{-2}\Vert_2\Vert A_n\Vert_2}{\delta_3}\left( \frac{108nC_n^2\Vert P_n\Vert_{F}^2}{\sqrt{N}}+9C_n^2\Vert P_n\Vert_2^2\sqrt{\frac{2\log(4/\epsilon)}{N}}+\delta\right)\label{eq:re1}.
\end{align}
The last inequality is derived using Markov's inequality \cite{lin2010probability} and the generalization error estimation in Theorem \ref{theorem:generalization_error}, and it holds with probability at least $1-\epsilon/2$.
Therefore, by choosing sufficiently large $N$ and solving the empirical loss function sufficiently well such that $\delta$ is small enough, we can have
$$ \text{Prob}(\Vert \mathcal{N}_{\theta^*}(\mathbf r)-A_n^{-1}\mathbf r\Vert^2_{A_n}\geq \delta_3\Vert A_n^{-1}\mathbf r\Vert^2_{A_n})\leq \epsilon/2, $$
indicating that
\begin{equation}
    \Vert \mathcal{N}_{\theta^*}(\mathbf r)-A_n^{-1}\mathbf r\Vert^2_{A_n}\leq \delta_3\Vert A_n^{-1}\mathbf r\Vert^2_{A_n}\label{eq:re2}
\end{equation}
with probability at least $1-\epsilon/2$.
The inequality $\Vert \mathcal{N}_{\theta^*}(\mathbf r)-A_n^{-1}\mathbf r\Vert^2_{A_n}\leq \delta_3\Vert A_n^{-1}\mathbf r\Vert^2_{A_n}$ hold if both (\ref{eq:re1}) and (\ref{eq:re2}) hold simultaneously. By \cite{kwerel1975bounds}, we have 
\[\text{Prob}(\Vert \mathcal{N}_{\theta^*}(\mathbf r)-A_n^{-1}\mathbf r\Vert^2_{A_n}\leq \delta_3\Vert A_n^{-1}\mathbf r\Vert^2_{A_n})\geq (1-\epsilon/2)+(1-\epsilon/2)-1=1-\epsilon.\]
\end{proof}

\section{Network structures} \label{sec:network_structure}
In Section \ref{sec:exp}, we adopt the Fourier neural operators (FNO) \cite{li2021fourier}. The general structure is shown in Figure \ref{fig:FNO}. Here, the first fully connected layer $L_1$ aims to lift the input dimension to $\mathbb{R}^h$ where $h$ is the hidden dimension specified by users. The last fully connected layer $L_2$ aims to change the output dimension back to $\mathbb{R}^1$ or $\mathbb{R}^2$. The Fourier layer is introduced in \cite{li2021fourier} and is defined as
$\sigma(\text{FFT}^{-1}\circ R_l\circ\text{FFT}(\mathbf z)+W_l\mathbf z+\mathbf b),\ l=1,\dots,L,$
where $\sigma$ is a nonlinear activation function, $R_l$ is a complex linear transformation applied only to the first $k$ low-frequency components in each dimension, $W_l$ is a linear transformation, and $\mathbf b$ is a bias vector. More details of the implementation of the FNO can be found in \cite{li2021fourier}.
Here, $W_l$ is implemented as a convolution operator with learnable convolutional kernels. The hyperparameters of networks used for Section \ref{sec:exp} are listed in Table \ref{tab:network_strucutres}.

\begin{table}[h]
\centering
\begin{tabular}{|c|c|cccc|}
\hline
problem dimension & problem size & $h$ & $L$ & $k$ & kernel size\\ \hline
\multirow{3}{*}{$d=1$}                  & $n=256$        &  64   &  6   &  64   &     7                    \\
                                      & $n=512$        & 64    &  6  &  128   &  7                           \\
                                      & $n=1024$       & 64    &  6   & 128    &   7                         \\ \hline
\multirow{3}{*}{$d=2$}                  & $n=256$        &  64   &  5   &   32  &      5                   \\
                                      & $n=512$        &  64   &   6  & 64    &    7                         \\
                                      & $n=1024$       &  64   &  6   &  64   &    7                  \\\hline  
\end{tabular}
\caption{Hyperparameters in neural operators}
\label{tab:network_strucutres}
\end{table}

\begin{figure}
    \centering
    \includegraphics[width=1.\linewidth]{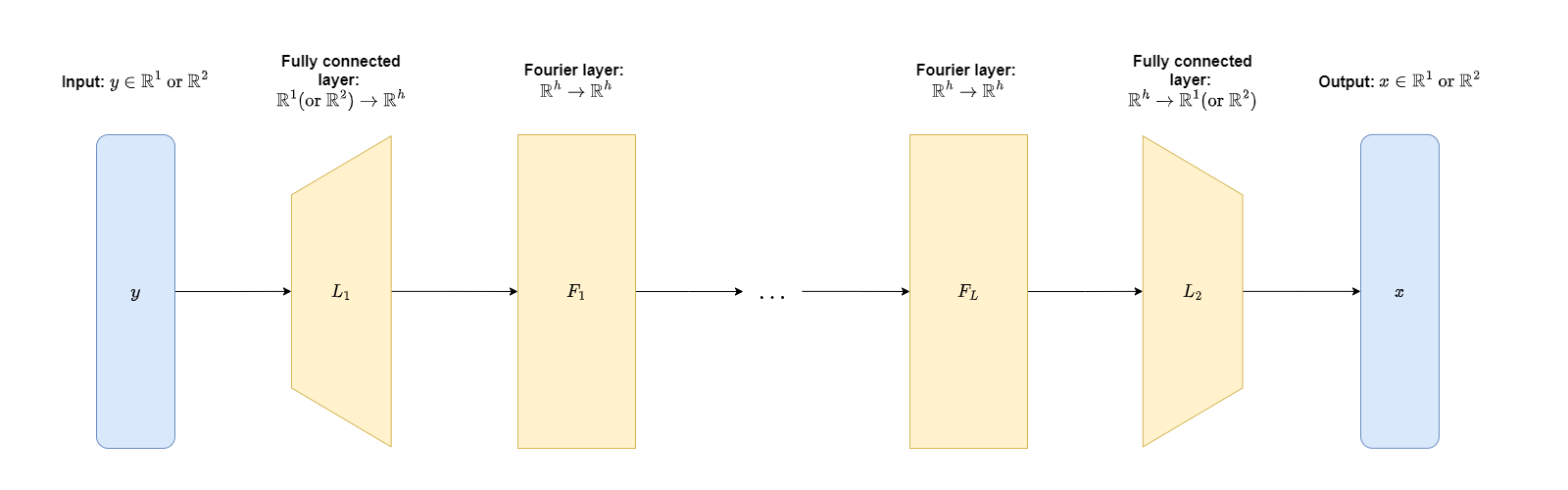}
    \caption{The structure of Fourier neural operators}
    \label{fig:FNO}
\end{figure}

\section{Table \ref{tab:condNum}: Condition number of linear systems used in Section \ref{sec:exp}}
\begin{table}[]
\centering
\begin{tabular}{|>{\centering}p{0.22\textwidth}|>{\centering}p{0.15\textwidth}|>{\centering}p{0.15\textwidth}|>{\centering}p{0.15\textwidth}|>{\centering}p{0.15\textwidth}|}
\hline
 & $\alpha$ &  $n=256$ & $n=512$ & $n=1024$ \tabularnewline \hline
 \multirow{3}{*}{\makecell{Section \ref{sec:exp_1d_tik}}}& $10^{-3}$ & 971    &  972   & 972    \tabularnewline \cline{2-5}
 & $10^{-4}$ & 7678    &   7683  & 7685    \tabularnewline \cline{2-5}
 & $10^{-5}$ & 24849    & 24904    & 24918    \tabularnewline \hline
 \multirow{3}{*}{\makecell{Section \ref{sec:exp_1d_ani}}}& $10^{-3}$ & 488    &  490   & 492    \tabularnewline \cline{2-5}
 & $10^{-4}$ & 4225    &   4286  & 4316    \tabularnewline \cline{2-5}
 & $10^{-5}$ & 19159    & 19546    & 19746    \tabularnewline \hline
 \multirow{3}{*}{\makecell{Section \ref{sec:exp_2d}}} & $10^{-3}$ & 999.7    & 999.9 & 1000.0    \tabularnewline \cline{2-5}
 & $10^{-4}$ & 9996.5    &   9999.1  & 9999.7    \tabularnewline \cline{2-5}
 & $10^{-5}$ & 99957.2    & 99982.5    & 99988.8    \tabularnewline \hline  
\end{tabular}
\caption{Condition number of linear systems corresponding to different choices of $\alpha$ and $n$ in Section \ref{sec:exp}}
\label{tab:condNum}
\end{table}

\end{document}

%% file: ex_shared.tex
\usepackage{lipsum}
\usepackage{amsmath,amsfonts,amssymb}
\usepackage{bbm}
\usepackage{stackengine}
\usepackage{algorithm}
\usepackage{algpseudocode}
\usepackage{listings}
\usepackage{booktabs}
\usepackage{epsfig}
\usepackage{color} 
\usepackage{graphicx}
\usepackage{subfigure}
\usepackage{float}
\usepackage{hyperref}
\usepackage{footnote}
\usepackage{mathrsfs}
\usepackage{booktabs}
\usepackage{listings}
\usepackage{multirow}
\usepackage{caption}
\usepackage[section]{placeins}
\usepackage{tabularx}
\usepackage{multirow}
\usepackage{makecell}
\usepackage{array}
\usepackage{epstopdf}
\ifpdf
  \DeclareGraphicsExtensions{.eps,.pdf,.png,.jpg}
\else
  \DeclareGraphicsExtensions{.eps}
\fi

\newsiamremark{remark}{Remark}
\newsiamremark{hypothesis}{Hypothesis}
\crefname{hypothesis}{Hypothesis}{Hypotheses}
\newsiamthm{claim}{Claim}

\headers{R. Chan and L. Li}{}

\title{Solving Convolution-type Integral Equations using Preconditioned Neural Operators\thanks{Submitted to the editors DATE.
\funding{This work of R. Chan is partially supported by HKRGC Grants No.  CityU11309922, LU13300125, ITF Grant No. MHP/054/22,  and LU BGR 105824.
The work of Lingfeng Li is supported by the InnoHK project at the Hong Kong Centre for Cerebro-cardiovascular Health Engineering (COCHE) and  HKRGC Grants No.LU13300125.}}}

\author{Raymond Hon-Fu Chan\thanks{Department of Data Science and Department of Operations and Risk Management, Lingnan University, Hong Kong; Hong Kong Centre for Cerebro-Cardiovascular Health Engineering, Hong Kong
  (\email{raymond.chanln.edu.hk}).}
\and Lingfeng Li\thanks{Hetao Institute of Mathematics and Interdisciplinary Sciences (Shenzhen), Shenzhen, China 
  (\email{lilingfeng@himis-sz.cn}).}}

\usepackage{amsopn}

\makeatletter
\newcommand*{\addFileDependency}[1]{
  \typeout{(#1)}
  \@addtofilelist{#1}
  \IfFileExists{#1}{}{\typeout{No file #1.}}
}
\makeatother

\newcommand*{\myexternaldocument}[1]{%
    \externaldocument{#1}%
    \addFileDependency{#1.tex}%
    \addFileDependency{#1.aux}%
}

%% file: bibfile.bib
@article{chan1994circulant,
  title={Circulant preconditioners for {T}oeplitz-block matrices},
  author={Chan, Tony F and Olkin, Julia A},
  journal={Numerical Algorithms},
  volume={6},
  number={1},
  pages={89--101},
  year={1994},
  publisher={Springer}
}

@article{chan1988optimal,
  title={An optimal circulant preconditioner for {T}oeplitz systems},
  author={Chan, Tony F},
  journal={SIAM Journal on Scientific and Statistical Computing},
  volume={9},
  number={4},
  pages={766--771},
  year={1988},
  publisher={SIAM}
}

@incollection{ruge1987algebraic,
  title={Algebraic multigrid},
  author={Ruge, John W and St{\"u}ben, Klaus},
  booktitle={Multigrid Methods},
  pages={73--130},
  year={1987},
  publisher={SIAM}
}

@article{xu2017algebraic,
  title={Algebraic multigrid methods},
  author={Xu, Jinchao and Zikatanov, Ludmil},
  journal={Acta Numerica},
  volume={26},
  pages={591--721},
  year={2017},
  publisher={Cambridge University Press}
}

@book{Shalev-Shwartz_Ben-David_2014, 
place={Cambridge}, 
title={Understanding Machine Learning: From Theory to Algorithms}, publisher={Cambridge University Press}, 
author={Shalev-Shwartz, Shai and Ben-David, Shai}, 
year={2014}}

@article{chui1982application,
  title={Application of approximation theory methods to recursive digital filter design},
  author={Chui, C and Chan, A},
  journal={IEEE Transactions on Acoustics, Speech, and Signal Processing},
  volume={30},
  number={1},
  pages={18--24},
  year={1982},
  publisher={IEEE}
}

@article{grenander1958toeplitz,
  title={Toeplitz Forms and Their Applications},
  author={Grenander, U},
  journal={California Monographs in Mathematical Sciences/University of California Press},
  year={1958}
}

@book{king1989digital,
  title={Digital filtering in one and two dimensions: design and applications},
  author={King, Robert and Ahmadi, Majid and Gorgui-Naguib, Raouf and Kwabwe, Alan and Azimi-Sadjadi, Mahmood},
  year={1989},
  publisher={Springer}
}

@article{tyrtyshnikov1992optimal,
  title={Optimal and superoptimal circulant preconditioners},
  author={Tyrtyshnikov, Evgenij E},
  journal={SIAM Journal on Matrix Analysis and Applications},
  volume={13},
  number={2},
  pages={459--473},
  year={1992},
  publisher={SIAM}
}

@article{hu2024hybrid,
  title={A hybrid iterative method based on {MIONet} for PDEs: Theory and numerical examples},
  author={Hu, Jun and Jin, Pengzhan},
  journal={arXiv preprint arXiv:2402.07156},
  year={2024}
}

@article{zhang2022hybrid,
  title={A hybrid iterative numerical transferable solver ({HINTS}) for {PDEs} based on deep operator network and relaxation methods},
  author={Zhang, Enrui and Kahana, Adar and Turkel, Eli and Ranade, Rishikesh and Pathak, Jay and Karniadakis, George Em},
  journal={arXiv preprint arXiv:2208.13273},
  year={2022}
}

@article{lu2021learning,
  title={Learning nonlinear operators via {DeepONet} based on the universal approximation theorem of operators},
  author={Lu, Lu and Jin, Pengzhan and Pang, Guofei and Zhang, Zhongqiang and Karniadakis, George Em},
  journal={Nature Machine Intelligence},
  volume={3},
  number={3},
  pages={218--229},
  year={2021},
  publisher={Nature Publishing Group UK London}
}

@article{xu2020frequency,
  title={Frequency Principle: Fourier Analysis Sheds Light on Deep Neural Networks},
  author={Xu, Zhi-Qin John},
  journal={Communications in Computational Physics},
  volume={28},
  number={5},
  pages={1746--1767},
  year={2020}
}

@inproceedings{
li2021fourier,
title={Fourier Neural Operator for Parametric Partial Differential Equations},
author={Zongyi Li and Nikola Borislavov Kovachki and Kamyar Azizzadenesheli and Burigede liu and Kaushik Bhattacharya and Andrew Stuart and Anima Anandkumar},
booktitle={International Conference on Learning Representations},
year={2021},
url={https://openreview.net/forum?id=c8P9NQVtmnO}
}

@article{huang2022learning,
  title={Learning optimal multigrid smoothers via neural networks},
  author={Huang, Ru and Li, Ruipeng and Xi, Yuanzhe},
  journal={SIAM Journal on Scientific Computing},
  volume={45},
  number={3},
  pages={S199--S225},
  year={2022},
  publisher={SIAM}
}

@inproceedings{chan1997multigrid,
  title={Multigrid for Differential-Convolution Problems Arising in Image Processing},
  author={Chan, Raymond and Chan, Tong F and Wan, WL},
  booktitle={6th Workshop on Scientific Computing},
  pages={58--72},
  year={1997},
  organization={Springer-Verlag}
}

@article{ammar1988superfast,
  title={Superfast solution of real positive definite {Toeplitz} systems},
  author={Ammar, Gregory S and Gragg, William B},
  journal={SIAM Journal on Matrix Analysis and Applications},
  volume={9},
  number={1},
  pages={61--76},
  year={1988},
  publisher={SIAM}
}

@article{bitmead1980asymptotically,
  title={Asymptotically fast solution of {Toeplitz} and related systems of linear equations},
  author={Bitmead, Robert R and Anderson, Brian DO},
  journal={Linear Algebra and its Applications},
  volume={34},
  pages={103--116},
  year={1980},
  publisher={Elsevier}
}

@article{chan1992circulant,
  title={Circulant preconditioners constructed from kernels},
  author={Chan, Raymond H and Yeung, Man-Chung},
  journal={SIAM Journal on Numerical Analysis},
  volume={29},
  number={4},
  pages={1093--1103},
  year={1992},
  publisher={SIAM}
}

@article{sun1997note,
  title={A note on the convergence of the two-grid method for {Toeplitz} systems},
  author={Sun, Hai-Wei and Chan, Raymond H and Chang, Qian-Shun},
  journal={Computers \& Mathematics with Applications},
  volume={34},
  number={1},
  pages={11--18},
  year={1997},
  publisher={Elsevier}
}

@article{wang2021learning,
  title={Learning the solution operator of parametric partial differential equations with physics-informed {DeepONets}},
  author={Wang, Sifan and Wang, Hanwen and Perdikaris, Paris},
  journal={Science Advances},
  volume={7},
  number={40},
  pages={eabi8605},
  year={2021},
  publisher={American Association for the Advancement of Science}
}

@article{chan1996conjugate,
  title={Conjugate gradient methods for {Toeplitz} systems},
  author={Chan, Raymond H and Ng, Michael K},
  journal={SIAM Review},
  volume={38},
  number={3},
  pages={427--482},
  year={1996},
  publisher={SIAM}
}

@article{chan1989spectrum,
  title={The spectrum of a family of circulant preconditioned {Toeplitz} systems},
  author={Chan, Raymond H},
  journal={SIAM Journal on Numerical Analysis},
  volume={26},
  number={2},
  pages={503--506},
  year={1989},
  publisher={SIAM}
}

@inproceedings{he2016deep,
  title={Deep residual learning for image recognition},
  author={He, Kaiming and Zhang, Xiangyu and Ren, Shaoqing and Sun, Jian},
  booktitle={Procedings of IEEE/CVF Computer Vision and Pattern Recognition Conference},
  pages={},
  year={770--778, 2016}
}

@inproceedings{ronneberger2015u,
  title={{U-Net}: Convolutional networks for biomedical image segmentation},
  author={Ronneberger, Olaf and Fischer, Philipp and Brox, Thomas},
  booktitle={Procedings of Medical Image Computing and Computer Assisted Intervention},
  pages={},
  year={234--241, 2015},
}

@article{Zhang2024BlendingNO,
  title={Blending neural operators and relaxation methods in {PDE} numerical solvers},
  author={Enrui Zhang and Adar Kahana and Eli Turkel and Rishikesh Ranade and Jay Pathak and George Em Karniadakis},
  journal={Nature Machine Intelligence},
  year={2024},
  url={https://doi.org/10.1038/s42256-024-00910-x}
}

@article{li2024physics,
author = {Li, Zongyi and Zheng, Hongkai and Kovachki, Nikola and Jin, David and Chen, Haoxuan and Liu, Burigede and Azizzadenesheli, Kamyar and Anandkumar, Anima},
title = {Physics-Informed Neural Operator for Learning Partial Differential Equations},
journal={ACM/JMS Journal of Data Science},
year = {2024},
issue_date = {September 2024},
publisher = {Association for Computing Machinery},
address = {New York, NY, USA},
volume = {1},
number = {3},
url = {https://doi.org/10.1145/3648506},
doi = {10.1145/3648506}}

@article{goswami2022physics,
  title={A physics-informed variational DeepONet for predicting crack path in quasi-brittle materials},
  author={Goswami, Somdatta and Yin, Minglang and Yu, Yue and Karniadakis, George Em},
  journal={Computer Methods in Applied Mechanics and Engineering},
  volume={391},
  pages={114587},
  year={2022},
  publisher={Elsevier}
}

@book{ledoux2013probability,
  title={Probability in {Banach} Spaces: Isoperimetry and processes},
  author={Ledoux, Michel and Talagrand, Michel},
  year={2013},
  publisher={Springer Science \& Business Media}
}

@book{chan1990circulant,
  title={Circulant preconditioners for elliptic problems},
  author={Chan, Raymond H and Chan, Tony F},
  year={1990},
  publisher={Department of Mathematics, University of California, Los Angeles}
}

@article{song2025physics,
  title={Physics-informed multi-grid neural operator: theory and an application to porous flow simulation},
  author={Song, Suihong and Mukerji, Tapan and Zhang, Dongxiao},
  journal={Journal of Computational Physics},
  volume={520},
  pages={113438},
  year={2025},
  publisher={Elsevier}
}

@book{lin2010probability,
  title={Probability Inequalities},
  author={Lin, Zhengyan},
  year={2010},
  publisher={Springer}
}

@article{chan1992family,
  title={A family of block preconditioners for block systems},
  author={Chan, Raymond H and Jin, Xiao-Qing},
  journal={SIAM Journal on Scientific and Statistical Computing},
  volume={13},
  number={5},
  pages={1218--1235},
  year={1992},
  publisher={SIAM}
}

@article{beck2009fast,
  title={A fast iterative shrinkage-thresholding algorithm for linear inverse problems},
  author={Beck, Amir and Teboulle, Marc},
  journal={SIAM Journal on Imaging Sciences},
  volume={2},
  number={1},
  pages={183--202},
  year={2009},
  publisher={SIAM}
}

@book{bergheau2013finite,
  title={Finite element simulation of heat transfer},
  author={Bergheau, Jean-Michel and Fortunier, Roland},
  year={2013},
  publisher={John Wiley \& Sons}
}

@inproceedings{rahaman2019spectral,
  title={On the spectral bias of neural networks},
  author={Rahaman, Nasim and Baratin, Aristide and Arpit, Devansh and Draxler, Felix and Lin, Min and Hamprecht, Fred and Bengio, Yoshua and Courville, Aaron},
  booktitle={International conference on machine learning},
  pages={5301--5310},
  year={2019},
  organization={PMLR}
}

@inproceedings{cao2021towards,
  title={Towards Understanding the Spectral Bias of Deep Learning},
  author={Cao, Yuan and Fang, Zhiying and Wu, Yue and Zhou, Ding-Xuan and Gu, Quanquan},
  booktitle={30th International Joint Conference on Artificial Intelligence (IJCAI 2021)},
  pages={2205--2211},
  year={2021},
  organization={International Joint Conferences on Artificial Intelligence}
}

@book{golub2013matrix,
  title={Matrix computations},
  author={Golub, Gene H and Van Loan, Charles F},
  year={2013},
  publisher={JHU press}
}

@article{liu2024preconditioning,
  title={Preconditioning for physics-informed neural networks},
  author={Liu, Songming and Su, Chang and Yao, Jiachen and Hao, Zhongkai and Su, Hang and Wu, Youjia and Zhu, Jun},
  journal={arXiv preprint arXiv:2402.00531},
  year={2024}
}

@book{polyanin2008handbook,
  title={Handbook of integral equations},
  author={Polyanin, Polyanin and Manzhirov, Alexander V},
  year={2008},
  publisher={Chapman and Hall/CRC}
}

@article{li2024priori,
  title={A priori error estimate of deep mixed residual method for elliptic {PDEs}},
  author={Li, Lingfeng and Tai, Xue-Cheng and Yang, Jiang and Zhu, Quanhui},
  journal={Journal of Scientific Computing},
  volume={98},
  number={2},
  pages={44},
  year={2024},
  publisher={Springer}
}

@inproceedings{he2024mgno,
  title={MgNO: Efficient Parameterization of Linear Operators via Multigrid},
  author={He, Juncai and Liu, Xinliang and Xu, Jinchao},
  booktitle={12th International Conference on Learning Representations, ICLR 2024},
  year={2024}
}

@book{hemker1984some,
  title={Some implementations of multigrid linear system solvers},
  author={Hemker, Pieter Wilhelm and de Zeeuw, Paulus Maria},
  year={1984},
  publisher={Stichting Mathematisch Centrum}
}

@book{briggs2000multigrid,
  title={A multigrid tutorial},
  author={Briggs, William L and Henson, Van Emden and McCormick, Steve F},
  year={2000},
  publisher={SIAM}
}

@article{chen2022meta,
  title={Meta-mgnet: Meta multigrid networks for solving parameterized partial differential equations},
  author={Chen, Yuyan and Dong, Bin and Xu, Jinchao},
  journal={Journal of Computational Physics},
  volume={455},
  pages={110996},
  year={2022},
  publisher={Elsevier}
}

@article{kwerel1975bounds,
  title={Bounds on the probability of the union and intersection of m events},
  author={Kwerel, Seymour M},
  journal={Advances in Applied Probability},
  volume={7},
  number={2},
  pages={431--448},
  year={1975},
  publisher={Cambridge University Press}
}

@article{weinan2019priori,
  title={A priori estimates of the population risk for two-layer neural networks},
  author={Weinan, E and Ma, Chao and Wu, Lei},
  journal={Communications in Mathematical Sciences},
  volume={17},
  number={5},
  pages={1407--1425},
  year={2019},
  publisher={International Press of Boston, Inc.}
}

@inproceedings{lu2021priori,
  title={A priori generalization analysis of the deep Ritz method for solving high dimensional elliptic partial differential equations},
  author={Lu, Yulong and Lu, Jianfeng and Wang, Min},
  booktitle={Conference on learning theory},
  pages={3196--3241},
  year={2021},
  organization={PMLR}
}

@article{ohn2019smooth,
  title={Smooth function approximation by deep neural networks with general activation functions},
  author={Ohn, Ilsang and Kim, Yongdai},
  journal={Entropy},
  volume={21},
  number={7},
  pages={627},
  year={2019},
  publisher={MDPI}
}


%% file: references.bib
@Misc{amsmath,
  author =	 {{American Mathematical Society}},
  title =	 {User's Guide for the \texttt{amsmath} Package
                  (Version 2.0)},
  url =		 {ftp://ftp.ams.org/pub/tex/doc/amsmath/amsldoc.pdf},
  urldate =	 {2015-07-30},
  year =	 2002}
